\documentclass{nmd/nmd} 

\title{Injectivity radii of hyperbolic  \\ integer homology 3-spheres}

\author{Jeffrey F. Brock}
\givenname{Jeffrey}
\surname{Brock}
\address{ Dept.~of Math.\\ 
Brown University \\ 
Box 1917 \\ 
Providence, RI 02912, USA
}
\email{jeff\_brock@brown.edu}
\urladdr{http://www.math.brown.edu/~brock/}

\author{Nathan M. Dunfield}
\givenname{Nathan}
\surname{Dunfield}
\address{ Dept.~of Math., MC-382 \\
          University of Illinois \\
          1409 W. Green St. \\
          Urbana, IL 61801, USA
}
\email{nathan@dunfield.info}
\urladdr{http://dunfield.info}

\arxivreference{}
\arxivpassword{}

\subjclass[2000]{} 

\newcommand{\core}{\mathrm{core}}
\newcommand{\tors}{\mathrm{torsion}}
\newcommand{\reg}{\mathrm{reg}}
\newcommand{\thin}[2]{\mathrm{thin}_{#2} \, #1}

\newcommand{\Mfcover}{\widetilde{M}_f}
\newcommand{\half}{\mathbb{H}}
\newcommand{\cx}{\mathbb{C}}
\newcommand{\compos}{\circ}
\newcommand{\detprime}{\mathrm{det'}}
\newcommand{\Sengun}{{\c{S}}eng{\"u}n}
\DeclareMathOperator{\TR}{TorRat}
\newcommand{\p}{\mathfrak{p}}
\newcommand{\q}{\mathfrak{q}}
\newcommand{\qbar}{\overline{\q}}

\newcommand{\tilomega}{\widetilde{\omega}}
\newcommand{\HS}{\textit{HS}}

\begin{document}

\begin{abstract} 
  We construct hyperbolic integer homology 3-spheres where the
  injectivity radius is arbitrarily large for nearly all points of the
  manifold.  As a consequence, there exists a sequence of closed
  hyperbolic \3-manifolds which Benjamini-Schramm converge to $\H^3$
  whose normalized Ray-Singer analytic torsions do \emph{not} converge
  to the $L^2$-analytic torsion of $\H^3$.  This contrasts with the
  work of Abert et.~al.~who showed that Benjamini-Schramm convergence
  forces convergence of normalized betti numbers.  Our results
  shed light on a conjecture of Bergeron and Venkatesh on the growth of
  torsion in the homology of arithmetic hyperbolic \3-manifolds, and
  we give experimental results which support this and related conjectures.
\end{abstract}
\maketitle

\section{Introduction}

By Mostow rigidity, a hyperbolic structure on a closed \3-manifold $M$
is unique up to isometry.  While the geometry of $M$ is thus
completely determined by its underlying topology, it can be difficult
to understand the qualitative and quantitative connections between
these two facets of $M$.  Here, we show that a geometric property
involving injectivity radii can be varied independently of the
homology of the manifold.  To state our results, we first need some
notation. The injectivity radius $\inj_x(M)$ at $x \in M$ is the
largest radius for which the ball about $x$ is embedded, and the
(lower) injectivity radius of $M$ itself is $\inj(M) = \inf \setdef{
  \inj_x(M) }{x \in M}$.  On the topological side, an \emph{integer
  homology \3-sphere} is a closed \3-manifold $M$ where $H_*(M; \Z)
\cong H_*(S^3; \Z)$, and the term \emph{rational homology 3-sphere} is
similarly defined.  Our main result here is:
\begin{theorem}\label{thm:IHS}
  Given positive constants $R$ and $\epsilon$ there exists a
  hyperbolic integer
  homology \3-sphere $M$ where
  \[
  \frac{\vol \Big( \setdef{ x \in M }{ \inj_x(M) < R } \Big)}{\vol(M)}  < \epsilon.
  \]
\end{theorem}
\noindent
In fact, we show that the homology of $M$ can be specified 
arbitrarily (Theorem~\ref{thm:full}).  The proof is based on the
modern theory of Kleinian groups; before sketching it, we motivate
our result in several ways.

\subsection{Cooper's question}

Starting with any closed hyperbolic \3-manifold, one can make the
injectivity radius arbitrarily large everywhere by taking a suitable
finite cover.  In the context of the Virtual Haken Conjecture, this
motivated Cooper to ask whether there are hyperbolic \emph{rational}
homology \3-spheres with arbitrarily large injectivity radius.  In
fact, such manifolds do exist by the work of Calegari-Dunfield and
Boston-Ellenberg \cite{CalegariDunfield2006, BostonEllenberg2006}.
However, if one instead considers \emph{integer} homology \3-spheres,
then the analogous question is open; our Theorem~\ref{thm:IHS} answers
affirmatively a weakened version of this question.   The manifolds of
\cite{CalegariDunfield2006, BostonEllenberg2006} came from a tower of
congruence covers of a fixed base manifold, and it seems unlikely this
method would work for integer homology \3-spheres as we now describe.

\subsection{Torsion growth}\label{sec:context}

Recently, number theorists have become interested in torsion in
the homology of arithmetic groups \cite{BergeronVenkatesh2010,
  CalegariVenkatesh2012}.  Specifically, Bergeron and Venkatesh
posited the following as part of an intriguing general conjecture for
arithmetic lattices in semisimple Lie groups; in the present context
of hyperbolic 3-manifolds, Le independently formulated a closely
related conjecture, see \cite{Le2009} for details.  

\begin{conjecture}[name=\cite{BergeronVenkatesh2010}]\label{conj:BV}
  Let $M$ be a closed congruence arithmetic hyperbolic \3-manifold,
  and $M \leftarrow M_1 \leftarrow M_2 \leftarrow M_3 \leftarrow
  \cdots$ a tower of congruence covers where $\inj(M_n) \to \infty$.
  Then the size of the torsion subgroup of $H_1(M_n; \Z)$ grows
  exponentially in $\vol(M_n)$ and moreover
  \begin{equation}\label{eqn:l2tor}
  \lim_{n \to \infty} \frac{\log\big| H_1(M_n ; \Z)_\tors \big|}{\vol(M_n)} =
    \frac{1}{6 \pi}
  \end{equation}
\end{conjecture}
\noindent
In particular, if this conjecture holds then the approach of
\cite{CalegariDunfield2006, BostonEllenberg2006} which used exactly
such a tower to answer Cooper's question cannot be modified to prove
Theorem~\ref{thm:IHS}.

One of two key parts to Conjecture~\ref{conj:BV} is the expected
convergence of Ray-Singer analytic torsion in such a tower of covers.
More precisely, the logarithm of the analytic torsion of a Riemannian
manifold $M$ is
\[
\tau(M) = \frac{1}{2} \sum_{k=0}^{\dim M} (-1)^{k} \cdot k \cdot \log
\left( {\detprime}(\Delta_k) \right)
\]
where $\Delta_k$ is the Laplacian on smooth $k$-forms and $\detprime$
is the zeta-regularized product of nonzero eigenvalues (see
e.g.~\cite{Muller1978} for details).  Then for covers $M_n$ as in
Conjecture~\ref{conj:BV}, part of (\ref{eqn:l2tor}) is that one should
have
\begin{equation}\label{eqn:torconv}
\lim_{n \to \infty} \frac{ \tau(M_n) }{ \vol(M_n) } =  \tau^{(2)}( \H^3 ) 
\end{equation}
where $\tau^{(2)}(\H^3) = 1/6 \pi$ is the $L^2$-analytic torsion of
$\H^3$.  A corollary of Theorem~\ref{thm:IHS} is that one \emph{need
  not} have (\ref{eqn:torconv}) for a sequence $M_n$ of hyperbolic
\3-manifolds which Benjamini-Schramm converge to $\H^3$, which is a
natural geometric notion of convergence implied by the
hypotheses of Conjecture~\ref{conj:BV}.  As this
corollary was the primary motivation for this paper, we now discuss it
and its context in detail.

\subsection{Benjamini-Schramm convergence} \label{sec:BS}

For a manifold $M$, we define $\thin{M}{R} = \setdef{ x \in M }{
  \inj_x(M) < R }$.  Following \cite{ABBGNRS2012}, we say that a sequence
$\{M_n\}$ of closed hyperbolic \3-manifolds Benjamini-Schramm converge
to $\H^3$ if for all $R > 0$ one has $\vol\big( \thin{M_n}{R} \big)
\big/ \vol(M_n) \to 0$ as $n \to \infty$.  We emphasize here that the
$M_n$ may have no relationship with each other beyond being
hyperbolic; in particular, they need not be covers of a fixed manifold.
Despite this, Abert et.~al.~were able to show that this notion of
geometric convergence also implies convergence of part of the
topology:
\begin{theorem}[name=\cite{ABBGNRS2012}]\label{thm:localconv}
  Let $M_n$ be a sequence of closed hyperbolic \3-manifolds 
  which Benjamini-Schramm converge to $\H^3$.   Then
  \begin{equation}\label{eqn:l2betti}
  \lim_{n \to \infty} \frac{\dim H_1(M_n ; \Q)}{\vol(M_n)} = 0.
  \end{equation}
  
\end{theorem}
Here, the $0$ in the right-hand side of (\ref{eqn:l2betti}) should be
interpreted as the first $L^2$--betti number of $\H^3$, and the moral
of Theorem~\ref{thm:localconv} is that suitable local convergence of
the geometry of the $M_n$ leads to convergence of their normalized
betti numbers to the corresponding $L^2$--betti number of their common
universal cover.  Theorem~\ref{thm:localconv} generalizes results of
L\"uck and Lott \cite{Luck1994, Lott1992} which apply only to $M_n$ coming from finite
covers of a fixed manifold (as in Conjecture~\ref{conj:BV}).

A key consequence of Theorem~\ref{thm:IHS} is that Theorem~\ref{thm:localconv} does
not have an analog for analytic torsion:
\begin{corollary}\label{cor:tornonconverge}
  There exist closed hyperbolic \3-manifolds $M_n$ which Benjamini-Schramm
  converge to $\H^3$ where $\tau(M_n) \big/ \vol(M_n) \to 0$ as
  $n \to \infty$.  In particular, the limit is not $\tau^{(2)}(\H^3) = 1/6\pi$.
\end{corollary}
Thus, while the geometric notion of Benjamini-Schramm convergence is
enough to control the convergence of (normalized) betti-numbers to the
corresponding $L^2$ invariant of the limit, the same is not true for
torsion.

\subsection{Experimental results}\label{sec:expintro}

Corollary~\ref{cor:tornonconverge} limits how much one can
broaden Conjecture~\ref{conj:BV}, and in this narrow sense could be
taken as evidence against Conjecture~\ref{conj:BV}.  However, we
present here computational evidence which strongly supports
Conjecture~\ref{conj:BV} as well as certain generalizations to
nonarithmetic manifolds.  Our experiments complement prior work of
\Sengun\ \cite{Sengun2011, Sengun2012, SengunBanff} and Page
\cite{PageBanff}.  To frame our results, we need to expand on the
connection between Conjecture~\ref{conj:BV} and analytic torsion.
For a closed Riemannian \3-manifold, the Cheeger-M\"uller theorem
\cite{Cheeger1979, Muller1978} implies (see
e.g.~\cite[\S 5.1]{CalegariVenkatesh2012}) that 
\begin{equation}\label{eqn:CM}
\tau(M) = {\log \big| H_1(M;\Z)_{\mathrm{tor}}
  \big| }  - \log \big( \vol(M) \big) + 2 \log \big( \mbox{regulator
  of $H^1(N)$} \big)
\end{equation}
Here the regulator of $H^1(N)$ is the covolume of the lattice $H^1(N ;
\Z)$ in $H^1(N; \R)$, where the latter has the inner product coming
from its identification with the set of harmonic forms.  The first
part of Conjecture~\ref{conj:BV} is that $\tau(M_n)/\vol(M_n) \to
1/6\pi$ and the second is that $\log\big( \reg \ H^1(M_n) \big) \big/
\vol(M_n) \to 0$.  In Section~\ref{sec:exp}, we provide evidence in
favor of a broadening of the first part Conjecture~\ref{conj:BV} to
\emph{all} hyperbolic \3-manifolds:
\begin{conjecture}\label{conj:gentor}
  Let $M_n$ be covers of a fixed closed hyperbolic \3-manifold $M$
  which Benjamini-Schramm converge to $\H^3$.  Then $\tau(M_n)/\vol(M_n) \to
  1/6\pi$.
\end{conjecture}
In contrast, it is not expected that $\log\big( \reg \ H^1(M_n) \big)
\big/ \vol(M_n) \to 0$ for nonarithmetic manifolds; we give data in
support of this, see especially
Figure~\ref{nonarithG}. For arithmetic
manifolds, experiments of \Sengun\ \cite{SengunBanff} identified the
case of congruence covers of prime-power level as a place where such
convergence appears to be slowest, to the point where one hits
computational limits before getting convincing evidence for or against
Conjecture~\ref{conj:BV}.  In Section~\ref{sec:exp}, we investigate
several families of examples of this type.  While some of these remain
ambiguous, overall they provide additional evidence that $\log\big( \reg \ H^1(M_n)
\big) \big/ \vol(M_n) \to 0$ even in this case.

\subsection{Proof sketch}

Given a homeomorphism $f$ of a surface $S$ there are two natural
3-manifolds we can build from it.  One is the mapping torus $M_f$,
which fibers over the circle.  Alternatively, we can identify $S$ with
the boundary of a handlebody $H$ and consider the associated Heegaard
splitting: $\HS_f = H \cup_f H$.  While the natural copies of $S$ in
$M_f$ and $\HS_f$ are radically different topologically (the first is
incompressible and the other maximally compressible), the philosophy
of Kleinian groups, specifically \cite{Namazi2005, NamaziSouto2009},
indicates that in favorable conditions on $f$, and for large powers
$n$, there are large chunks of the geometry of $M_{f^n}$ and
$\HS_{f^n}$ that are nearly isometric.

Here is the basic idea behind the manifolds in Theorem~\ref{thm:IHS}.
Fixing $R > 0$, it is easy to construct $(S, f)$ so that $M_f$ has
$\inj(M_f) > R + 1$.  Now for $M_f$ we have $b_1(M_f) > 0$, and in
particular $M_f$ is not a homology sphere.  However, we will
``photocopy'' its geometry onto a Heegaard splitting whose homology we
can independently control.  Specifically, choose homeomorphisms $h$
and $g$ of $S$ so that $\HS_h = S^3$ and $g$ acts trivially on $H_1(S;
\Z)$.  Then define $M_n$ to be the Heegaard splitting associated to $h \circ
f^n \circ g \circ f^{-n}$.  This $M_n$ is an integral homology sphere
since the gluing map acts on $H_1(S; \Z)$ precisely as $h$ does.  We
show that $f$ and $g$ can be chosen so that when $n$ is large, most of
the geometry of $M_n$ is locally nearly isometric to $M_f$ and hence
$\inj_x(M_n) > R$ on most of $M_n$.  Specifically, the volume of
$\thin{M_n}{R}$ is uniformly bounded whereas $\vol(M_n) \to \infty$;
hence we can make the ratio $\vol(\thin{M_n}{R})/\vol(M_n) <
\epsilon$, as required by Theorem~\ref{thm:IHS}.

In realizing this outline, there are several different routes one could take
through the machinery of Kleinian groups.  We choose one which only
uses results about manifolds with incompressible boundary and bounded
geometry.  Moreover, unlike the corresponding parts of
\cite{Namazi2005}, our argument does not rely on \cite{Tian1990}.  

\subsection{Open questions}
One very natural question is whether there are integral homology
3-spheres where the injectivity radius is large everywhere.  From the
point of view in the discussion in Sections~\ref{sec:context} and
\ref{sec:BS}, in fact it would be very interesting if one could add to
Theorem~\ref{thm:IHS} a \emph{uniform} lower bound on $\inj(M)$
independent of $R$ and $\epsilon$.  The current construction provides
no control on $\inj(M)$ as $R$ varies, basically because the genus of
$S$ has to change with $R$; see Remarks~\ref{rmk:genusinj} and
\ref{rmk:injepsilon}.

The weaker version of Theorem~\ref{thm:IHS} where one just requires
that $\inj_x M > R$ for \emph{some} $x$ follows from
\cite{PurcellSouto2010} by doing $1/n$ Dehn filling on the knot
complements constructed there which also have this property.  A
natural question is whether there are knots in $S^3$ where $\inj_x M$
is big most places. We give a possible construction of such knots in
Remark~\ref{rmk:knots}.

\subsection{Outline of the rest of the paper}

Section~\ref{sec:mainthm} gives the precise construction of the
manifolds in Theorem~\ref{thm:IHS} and proves that result modulo the
central Lemma~\ref{lem:geom-of-cusped}.  Section~\ref{sec:lemma}
reviews the needed background in Kleinian groups and uses it to prove
Lemma~\ref{lem:geom-of-cusped}.  Finally, Section~\ref{sec:exp}
contains the details of the experimental results.

\subsection{Acknowledgements}

The authors were partially supported by US NSF grants DMS-0906229,
DMS-0707136, and DMS-1105476.  We are very grateful to Nicolas
Bergeron for suggesting this question and explaining its relation to
\cite{ABBGNRS2012}, which happened at the conference ``Geometry,
analysis, and surfaces'' in Autrans, France, in March 2011.  The
computational part of this paper was motivated by a workshop on
torsion in the homology of arithmetic groups held in Banff in July
2012.  We thank the organizers of both of these excellent events.
Finally, we thank the referee for providing very helpful comments on
the initial version of this paper.

\section{Proof of the main theorem}\label{sec:mainthm}

The main result of this paper is:
\begin{theorem}\label{thm:full}
  Given positive constants $R$ and $\epsilon$ and a finitely-generated abelian
  group $A$, there exists a closed hyperbolic \3-manifold $M$ where
  \[
  H_1(M ; \Z) = A \mtext{and}  \frac{\vol(\thin{M}{R})}{\vol(M)}  < \epsilon.
  \]
\end{theorem}
\noindent
We begin by constructing a certain \3-manifold which fibers over the
circle, the mapping torus of a homeomorphism of a surface,
which will be used as the geometric model for most of the manifold in
Theorem~\ref{thm:full}.

\begin{lemma}\label{lem:mp-inj}
  Given $R > 0$, there exists a closed hyperbolic \3-manifold $M$
  which is a mapping torus where $\inj(M) > R$.  
\end{lemma}

\begin{proof}
  Fix some hyperbolic mapping torus $N$.  Then $N$ contains finitely many
  closed geodesics of length $\leq 2R$, corresponding to certain
  conjugacy classes $[\gamma_i]$ of elements of $\pi_1(N)$.  Since
  $\pi_1(N)$ is residually finite (see e.g.~\cite{LongReid1998}),
  there is a finite-index normal subgroup of $\pi_1(N)$
  which contains no $\gamma_i$.  If $M$ is the corresponding finite
  cover, then its shortest geodesic has length $> 2R$ and hence
  $\inj(M) > R$.  Since the fibration of $N$ over $S^1$ pulls back to
  one of $M$, we are done.  
\end{proof}

\begin{remark} \label{rmk:genusinj}
  A simple argument using minimal surfaces shows that any mapping
  torus of a surface of genus $g$ with $\inj(M) = R$ has $\log(g) \geq
  R - C$, where $C$ is independent of $R$; thus the genus of the fiber
  of $M$ in Lemma~\ref{lem:mp-inj} necessarily goes to infinity as $R$
  does.  While we have no need for this here, with a little more care
  the above construction can produce examples where $\log(g) \leq 3 R
  + C'$ as we now describe.  Specifically, take the base manifold $N$
  to be arithmetic of the simplest type, i.e.~defined by some
  quadratic form.  (There are many such fibered $N$ by Theorem 5.2 of
  \cite{Agol2008}.)  Now consider a tower $M_n$ of congruence covers
  of $N$.  If $d_n$ is the degree of $M_n \to N$, by Lemma 2.2.1 of
  \cite{Yeung1994} we know there is a constant $C''$ so that
  $\inj(M_n) \geq (1/3) \log d_n - C''$. On the other hand, the genus
  of the fiber grows at most linearly in $d_n$, and hence satisfies
  $\log(g) \leq 3 R + C'$ for some $C'$ independent of $R$.
\end{remark}

\subsection{Main Construction} \label{sec:construction}

We now detail the construction of the examples in
Theorem~\ref{thm:full}.  Throughout, fix $R > 0$ and a finitely
generated abelian group $A$.  Via Lemma~\ref{lem:mp-inj}, we choose a
pseudo-Anosov homeomorphism $f$ of a closed surface $S$ so that the
mapping torus $M_f$ has $\inj(M_f) > R + 1$.  Let $N_0$ be a connected
sum of lens spaces and copies of $S^2 \times S^1$ so that $H_1(N_0 ;
\Z) = A$.  Let $g$ be the genus of $S$, and let $H^+ \cup H^-$ be a
Heegaard splitting of $N_0$ of genus $g$; such a splitting exists
provided $g \geq \rank(A)$, and we can always make $g$ bigger if
necessary by replacing $M_f$ with a suitable finite cover.  Now
identify the Heegaard surface $\partial H^+ = \partial H^{-}$ with
$S$.  Choose a pants decomposition $P$ of $S$ so that the pared
manifolds $(H^\pm, P)$ are acylindrical; any $P$ at distance at least
$3$ from the disc sets of $H^+$ and $H^-$ will do.

Let $\gamma$ be a
separating essential simple closed curve on $S$ so that the pared manifold
\[
U = \big( (S \times [0, 2]) \setminus (\gamma \times \{1\}), \  P
\times \{0\} \cup \ P \times \{2\} \big)
\]
is acylindrical.  We now define a family of links in $N_0$ which
lie in a product neighborhood $S \times [0, 6]$ as follows
\[
L_n = P \times \{1\} \  \cup \ f^n(P) \times \{2\} \  \cup \ f^n(\gamma) \times
\{3\} \ \cup \ f^{n}(P) \times \{4\} \  \cup \ P \times \{5\}
\]
and consider their complements $N_n = N_0 \setminus L_n$.  We frame
$L_n$ by the blackboard framing with respect to the surfaces $S \times
\{s\}$ which contains it; that is, a longitude is a parallel copy of
the corresponding
component in $S \times \{s\}$. Define the closed manifold $N_{n,k}$ to be the following Dehn surgery on
$L_n$ in $N_0$: do $1/k$ Dehn surgery on each component which
is at heights $\{1, 2, 3\}$ and $-1/k$ Dehn surgery on each component at
heights $\{4, 5\}$.   For large $n$ and
$k$, these $N_{n,k}$ will be the examples used to prove Theorem~\ref{thm:full}.
To start, we show
\begin{lemma}\label{lem:hom-ok}
  The homology $H_1(N_{n,k} ; \Z) = A$ for all $n,k$. 
\end{lemma}
\begin{proof}
  Doing $1/k$ Dehn surgery along a single curve $\eta$ in $S$
  is equivalent to changing the gluing of the Heegaard splitting
  by the $k^{\mathit{th}}$ power of the Dehn twist on $\eta$.  Since
  $\gamma$ is separating, a Dehn twist on it acts trivially on the
  homology of $S$.  Thus, homologically, the Dehn twists along the
  components of $L_n$ at heights $\{1, 2\}$ precisely cancel out those
  at heights $\{4, 5\}$.  Hence $N_{n,k}$ has the same homology as $N_0$.  
\end{proof}

The key geometric claim is the following, whose proof we defer to Section~\ref{sec:lemma}.

\begin{lemma}\label{lem:geom-of-cusped}
  Let $\{N_n\}$ be the sequence of manifolds constructed above from
  the chosen $R > 0$.  For all large $n$, the manifold $N_n$ has a
  complete hyperbolic metric of finite volume, and moreover
  \[
  \lim_{n \to \infty} \frac{ \vol(\thin{N_n}{R})}{\vol(N_n)} = 0
  \]
\end{lemma}

\begin{proof}[Proof of Theorem~\ref{thm:full}]
  Let $\epsilon > 0$ be given.  By Lemma~\ref{lem:geom-of-cusped},
  choose $n$ large enough so that $N_n$ is hyperbolic and
  $\vol(\thin{N_n}{R})\big/\vol(N_n) < \epsilon/2$.  We now view
  $N_{n,k}$ as a Dehn filling on the cusped manifold $N_n$.  By
  Thurston's Hyperbolic Dehn Surgery Theorem, for large $k$ the
  manifold $N_{n,k}$ is hyperbolic; moreover, the geometry of
  $N_{n,k}$ is arbitrarily close to that of $N_n$ outside a set of
  arbitrarily small volume, which is a neighborhood about the core
  geodesics of the added solid tori \cite{ThurstonLectureNotes,
    PetronioPorti2000}.  In particular, we can choose $k$ so that
  $\vol(\thin{N_{n,k}}{R})\big/\vol(N_{n,k}) < \epsilon$.  Since
  $H_1(N_{n,k}; \Z) = A$ by Lemma~\ref{lem:hom-ok} we have proved the
  theorem.
\end{proof}

\begin{remark} \label{rmk:injepsilon}
  For fixed $R$, the manifolds used to prove Theorem~\ref{thm:full}
  can be chosen with minimum injectivity radius bounded below
  independent of $\epsilon$ as we now explain.  As shown in
  Section~\ref{sec:lemma}, for large $n$ the manifolds $N_n$
  constructed have injectivity radius uniformly bounded below outside
  neighborhoods of the cusps.  Moreover, the geometry of said cusps are
  nearly isometric for large $n$.  The Drilling Theorem
  \cite{BrockBromberg2004} then shows that the choice of $k$ so that
  $N_{n,k}$ has geometry close to that of $N_n$ can be made
  independent of $n$, and the added core geodesics in $N_{n,k}$ have
  length uniformly bounded from below.
\end{remark}

\begin{remark} \label{rmk:knots} We chose the construction here to
  streamline the proof of Lemma~\ref{lem:geom-of-cusped} in
  Section~\ref{sec:lemma}.  Here is a combinatorially simpler
  construction satisfying Lemma~\ref{lem:geom-of-cusped} that relies
  on work of Namazi in his (as yet unpublished) thesis
  \cite{Namazi2005}, the relevant results of which will appear in
  \cite{BrockMinskyNamaziSoutoInPrep}; we hew to the published
  literature in 
  our present treatment.
  Let $f$ be as before, but if necessary change the identification of
  $S$ with the Heegaard surface of $N_0$ so that the invariant laminations
  of $f$ are disjoint from the closure in $\PML(S)$ of the disk sets
  of both $H^+$ and $H^-$ (which can be done by
  \cite{Kerckhoff1990,Gadre2011}).  Once again letting $\gamma$ be a
  separating curve on $S$, take $N'_n$ simply to be $N_0 \setminus
  f^n(\gamma)$.  By a bounded geometry model theorem for Heegaard
  splittings \cite{Namazi2005, BrockMinskyNamaziSoutoInPrep} (similar
  to Minsky's bounded geometry theorem \cite{Minsky2001} in the
  I-bundle case), given a sufficiently large $k$,
  chosen independent of $n$, the geometry of a $1/k$ Dehn-filling of
  $N'_n$ will be modeled up to bi-Lipschitz distortion by the
  geometry of that of $M_f$ for almost all of its volume.  An exactly
  analogous argument to the one given in the proof of
  Theorem~\ref{thm:maxcusps} allows us to make the bi-Lipschitz
  constant arbitrarily close to $1$ for almost all of the volume.  In
  our current treatment, the extra pairs of pants used to define $N_n$
  give us many canonical thrice-punctured spheres which, because of
  their rigidity, are natural places from which to understand the
  overall geometry of $N_n$ via geometric limits.

\end{remark}

\section{Proof of the main lemma}\label{sec:lemma}

The proof of Lemma~\ref{lem:geom-of-cusped} is our point of entry into
the modern theory of Kleinian groups. We first isolate the necessary
background before turning to the proof itself. 

\subsection{Kleinian background}

Throughout Section~\ref{sec:lemma}, we take $S$ to be a closed surface of
genus $g > 1$.  We denote by $AH(S)$ the set of all complete
hyperbolic \3-manifolds $M = \half^3/\Gamma$ equipped with
\emph{markings}, or homotopy equivalences $h \colon S \to M$, up to
marking preserving isometry; precisely,
$$(h \colon S \to M) \sim (g \colon S \to N)$$
if there is an
isometry $\phi \colon M \to N$ where $\phi \circ h \simeq g$.
The {\em mapping class group} $\MCG(S)$ of orientation preserving
self-homeomorphisms of $S$ up to isotopy acts on $AH(S)$ by
precomposition: given $f \in \MCG(S)$ we let 
$$f \cdot (h \colon S \to M) = \big(h \circ f^{-1} \colon S \to
M\big).$$
We refer to this action as {\em remarking} the element $(h \colon S
\to M)$ by $f$. 

A hyperbolic 3-manifold $M$ determines a conjugacy class of
\emph{Kleinian groups}, that is, of discrete subgroups of
$\Isom^+(\H^3) = \PSL{2}{\C}$.  A specific group is identified by a
choosing once and for all a fixed {\em baseframe} $\tilomega$, that
is, an orthonormal frame in the tangent space at a point in $\H^3$,
and a baseframe $\omega$ in the tangent space at a point in $M$; the
group $\Gamma$ is then taken so that the derivative of the
covering projection $$\H^3 \to M = \H^3/\Gamma$$ sends $\tilomega$ to $\omega$.
In practice, we will refer to a base-frame $\omega$ as being {\em in}
$M$ in reference to the underlying basepoint.

The space $AH(S)$ is readily seen to be the set of conjugacy classes
of discrete faithful representations $\rho \colon \pi_1(S) \to
\PSL{2}{\mathbb{C}}$, via the association $[\rho] = h_*$; $AH(S)$ is
topologized by convergence of representatives on generators.

On the level of manifolds, we can reformulate algebraic
convergence: 
a sequence
$(h_n, M_n)$ of elements of $AH(S)$ \emph{converges algebraically} to
$(h,M)$ if for each compact subset $K \subset M$ there are smooth homotopy equivalences $\varphi_n \colon M
\to M_n$ with $\varphi_n \compos h \simeq h_n$ so that for each
compact subset $K \subset M$ the derivatives $D\varphi_n$ 
converge to an isometry at each point of $K$.
If a baseframe $\omega$ in $M$ is chosen so that $(M,\omega)$ has
corresponding Kleinian group $\Gamma$, then taking $K$ containing
$\omega$, the baseframes $\omega_n = D\varphi_n(\omega)$ in $M_n$
determine Kleinian groups $\Gamma_n$ admitting isomorphisms $\rho_n
\colon \pi_1(S) \to \Gamma_n$ that converge to a limit $\rho \colon
\pi_1(S) \to \PSL{2}{\C}$ in the sense that $\rho_n(\gamma) \to
\rho(\gamma)$ for all $\gamma \in \pi_1(S)$; here $\rho_n = (h_n)_*$
and $\rho = h_*$.

 Based manifolds $(M_n,
\omega_n)$ \emph{converge geometrically} to a {\em geometric limit} $(M_G,\omega_G)$ if their associated Kleinian
groups $\Gamma_n$ converge to the associated Kleinian group $\Gamma$
for $(M_G,\omega_G)$ in the {\em Hausdorff topology}:
\begin{enumerate}
\item for each $\gamma \in \Gamma$ there are $\gamma_n \in \Gamma_n$
  so that $\{\gamma_n\}_n \to \gamma$, and 
\item if $\gamma$ is a limit point in $\PSL{2}{\C}$ of a set
  $\{\gamma_n'\}_n$ with $\gamma_n' \in \Gamma_n$, then $\gamma$ lies
  in $\Gamma$.
\end{enumerate}
By elementary compactness results (see
\cite[Prop.~2.1]{McMullen1996}), any algebraically convergent sequence
$(h_n,M_n) \to (h, M)$ has a subsequence with an
associated geometric limit $M_G$; this geometric limit is
obtained by choosing baseframes $\omega_n$ to obtain convergent
representations $\rho_n \to 
\rho$ and then passing to a convergent subsequence of the corresponding
sequence of based manifolds $(M_n,\omega_n)$.

Note that we have a locally isometric covering map $(M,
\omega) \to (M_G, \omega_G)$.  The sequence $(h_n, M_n)$
\emph{converges strongly} if it converges both algebraically and
geometrically and moreover the locally isometric cover $M \to
M_G$ is an isometry (in particular, a {\em homeomorphism}).

Geometric convergence also has this intrinsic formulation:
$(M_n,\omega_n) \to (M_G, \omega_G)$ if for each compact subset $K
\subset M_G$ with $\omega_G \in K$, there are smooth bi-Lipschitz
embeddings 
$$\psi_n \colon (K,\omega_G) \to (M_n,\omega_n)$$ for $n$ sufficiently
large so that the derivatives of $\psi_n$ converge to isometries at
each point of $K$.
 While the limit $(M_G,\omega_G)$
depends on the choice of baseframes $\omega_n$, if $\omega_n'$ lie at
a uniformly bounded distance from $\omega_n$ then any limit of the
sequence $(M_n, \omega_n')$ is isometric to $M_G$.

We adopt the convention that given an algebraically convergent sequence
$$(h_n,M_n) \to (h,M)$$ and a choice of $\omega$ in $M$, that baseframes
$\omega_n$ are determined by the associated smooth homotopy
equivalences $\varphi_n \colon M \to M_n$ with via $D
\varphi_n(\omega) = \omega_n$.  With this convention, images
$\varphi_n\circ h(S)$ sit at uniformly bounded distance from the
baseframes $\omega_n$.

\subsection{Maximal Cusps} If $P$ and $Q$ are sets of simple closed
curves giving a pants decomposition of $S$, denote by $M(P, Q)$ the
corresponding pared manifold
\[
\big(S \times I, \ P \times \{0\} \cup Q \times \{1\}\big).
\]
We say $M(P, Q)$ is {\em pared acylindrical} if no simple closed curve
isotopic into $P$ is also isotopic into $Q$.  For pared acylindrical
$M(P,Q)$ there is a finite-volume hyperbolic structure on $S \times
\R$ so that each free homotopy class represented by the pared locus
corresponds to a rank-1 cusp.  The hyperbolic structure is unique, and
letting $S$ mark $M(P,Q)$ by its inclusion as $S \times \{1/2 \}$, we
obtain a boundary point in the deformation space $AH(S)$ known as a
\emph{maximal cusp}.

The convex core of $M = \half^3 /\Gamma$, denoted $\core(M)$, is the
quotient by $\Gamma$ of the smallest convex subset of $\half^3$ whose
closure contains the limit set of $\Gamma$, which is the intersection
of the closure of an orbit of $\Gamma$ with $\hat \cx = S^2_\infty$.
The pared convex core, written $\core^0(M)$, is the complement in
$\core(M)$ of its intersection with the Margulis thin parts of $M$
corresponding to cusps.  While $\core\big(M(P,Q)\big)$ has frontier
consisting of totally geodesic triply-punctured spheres, the boundary
of $\core^0\big(M(P,Q)\big)$ consists of a pair of compact surfaces
each containing a collection of distinguished annuli representing its
intersection with cusps corresponding to $P$ and $Q$ respectively.

Much of the theory of algebraic and geometric limits of quasi-Fuchsian
manifolds $Q(X,Y)$ in $AH(S)$ can be carried out for maximal cusps $M(P,Q)$ by
viewing the pair $(P,Q)$ as a combinatorial version of the pair
$(X,Y) \in \Teich(S) \times \Teich(S)$ of marked conformal 
structures determining $Q(X,Y)$.
Indeed, as each $M(P,Q)$ is uniquely determined by the
choice of $P$ and $Q$, much of the theory becomes more concrete in
this setting.

\subsection{Pseudo-Anosov double limits}  
For a pseudo-Anosov element $f\in \MCG(S)$, we fix a fiber $F$ in
the associated mapping torus $M_f$, the corresponding fibration over $S^1$ with
monodromy $f$.  We define the
\emph{block} $B_f$ of $f$ to be $M_f$ split open along $F$, that is,
the closure of $M_f \setminus F$ in the path metric.  We define
$\Mfcover$ to be the infinite-cyclic cover of $M_f$ corresponding to
$\pi_1(F)$.

Thurston and McMullen showed that the double iteration
$Q\left(f^{-n}(X),f^n(X)\right)$ of $f$ on quasi-Fuchsian
manifolds converges strongly to $\Mfcover$.
Likewise, McMullen established that the one-sided iteration
$Q\left(X,f^n(X)\right)$ converges strongly to a limit $Q_f$ with one end
asymptotically isometric to $\Mfcover$: there is a bi-Lipschitz
diffeomorphism between neighborhoods of the infinite-volume end of
$\core( Q_f )$ and an end of $\Mfcover$ so that the norm of the
derivative converges to $1$.   Each of these discussions can be
carried out in the setting of maximal cusps:
\begin{proposition}\label{prop:limits}
  The maximal cusps $M\left(f^{-m}(P),f^n(P)\right)$ for $m, n >0$
  converge strongly to $\Mfcover$ as $m, n \to \infty$.  The one-sided iteration $M\left(P,
    f^n(P)\right)$ converges strongly to a manifold $M_A$ 
  whose pared convex core contains one compact boundary surface $S$
  with parabolic locus $P$ and a degenerate end asymptotically
  isometric to the positive end of $\Mfcover$.  The analogous
  statement holds for 
  $M\left(f^{-n}(P),P \right)$, whose limit is denoted $M_C$. 
\end{proposition}
\noindent 
See Figure~\ref{fig:maxcusps} for schematic pictures of $M_A$ and $M_C$.

\begin{proof}[Proof sketch]
  There are various ways to deduce these results, which follow easily
  from variations of the original arguments in \cite{Thurston1998,
    McMullen1996}.  Perhaps the simplest is the following, where for
  concreteness we focus on the first claim.  Consider a surface $X \in
  \Teich(S)$ where $P$ has very short total length and apply the Drilling
  Theorem of \cite{BrockBromberg2004} to the short geodesic
  representatives of $f^{-m}(P)$ and $f^n(P)$ in the quasi-Fuchsian hyperbolic
  3-manifold $Q_{m,n} = Q\big(f^{-m}(X),f^n(X)\big)$. The drilled manifold
  $D_{m,n}$ has a bi-Lipschitz diffeomorphism between $\core^0(D_{m,n})$ and a
  subset of $Q_{m,n}$; this diffeomorphism can be made arbitrarily close
  to isometric by making the length of $P$ on $X$ small enough.
  Now since $D_{m,n}$ has a cover isometric to
  $M\left(f^{-m}(P),f^n(P)\right)$, a diagonal argument yields the
  proposition.
\end{proof}

Our main result of this section is:
\begin{theorem}\label{thm:maxcusps}
  Given a pseudo-Anosov $f \in \MCG(S)$ and a pants decomposition $P$
  of $S$, let $Y_n = M\left(f^{-n}(P), f^n(P)\right)$.  For each
  $\epsilon >0$ there are finite-volume hyperbolic \3-manifolds $A$
  and $C$ so that for all $n$ sufficiently large, $\core(Y_n)$
 has a decomposition
  $$ \core(Y_n)  = A_n \cup B_n \cup C_n$$
  where $A_n$ and $C_n$ are $1+\epsilon$ bi-Lipschitz to $A$ and $C$
  and $\inj_b(Y_n) > \inj(M_f) - \epsilon$ for every $b \in B_n$.
  Moreover $\vol(B_n) \to \infty$ as $n \to \infty$.
\end{theorem}

\begin{remark}
  The theory of Kleinian surface groups provides considerable
  information about the manifolds $Y_n$; in particular, Minsky's
  Bounded Geometry Theorem \cite{Minsky2001} guarantees there is a
  {\em bi-Lipschitz model} for $\core^0(Y_n)$ which can be described
  as a union of finitely many copies of $B_f$, and the bi-Lipschitz
  constant depends only on the genus of the fiber $F$ (we give a more
  detailed discussion in the proof of Theorem~\ref{thm:maxcusps}).
  Because we wish to ensure that the injectivity radius on $B_n$ is
  large, the dependence of the bi-Lipschitz constant on the genus
  presents a difficulty, as the lower bound for the injectivity radius
  of $M_f$ would then also depend on the genus of $F$.  Nevertheless
  we use this bi-Lipschitz control as a starting point.
\end{remark} 

Before proving Theorem~\ref{thm:maxcusps}, we explain its connection
to the geometry of the manifolds $N_n$ from
Section~\ref{sec:construction} and how it proves
Lemma~\ref{lem:geom-of-cusped}.  

\begin{proof}[Proof of Lemma~\ref{lem:geom-of-cusped}]  We return to
  the notation from Section~\ref{sec:construction}.  Let $M^\pm$ be
the convex cores of the manifolds corresponding to the pared
manifolds $(H^\pm, P)$.  Let $D$ be the convex core of the hyperbolic
manifold corresponding to $U$, and $D_n$ be its remarking by $f^n$,
i.e.~let $D_n$ be the convex core of the
pared manifold
\[
U_n = \big( (S \times I) \setminus (f^n(\gamma) \times \{1/2\}), \  f^n(P), \ f^n(P) \big).
\]
Then $N_n$ is the union of the following pieces,
glued along their totally geodesic surface boundaries (since these are all
thrice-punctured spheres there are no moduli issues):
\[
N_n = M^+ \ \cup \ \core{\big( M(P, f^n(P) ) \big)} \ \cup \ D_n \ \cup \
\core \big( M(f(P^n), P ) \big) \ \cup \  M^-.
\]
The geometries of $M^\pm$ and $D_n$ are fixed, and in particular so
are their volumes.  The other pieces are remarkings of the manifolds of
Theorem~\ref{thm:maxcusps}, and hence for large $n$ have injectivity
radius at least $\inj(M_f) - \epsilon$ outside a set of uniformly
bounded volume.  This proves Lemma~\ref{lem:geom-of-cusped}.
\end{proof}

\begin{proof}[Proof of Theorem~\ref{thm:maxcusps}.] 
  The mapping torus $M_f$ is defined as $S \times [0, 1]$ where $(x,
  1) \sim (f(x), 0)$.  The cover $\Mfcover$ is thus $S \times \R$
  where the deck group is generated by the self-isometry $\alpha$
  sending $(x, t)$ to $\big(f^{-1}(x), t + 1\big)$. 
We take our preferred fiber $F$ in $M_f$ to be $S \times \{0\}$, and
the default marking $h_0 \maps S \to \Mfcover$ to be the inclusion of
$S$ as $S \times \{0\}$.  Note that the action of $f$ on $AH(S)$
commutes with the action by $\alpha$:
$$\alpha \circ h_0 \simeq  f \cdot h_0 = h_0 \circ f^{-1}.$$
Further, we denote by $F_k$ the translate $\alpha^k(F) = S \times
\{k\}$ of the fiber; compare the top of Figure~\ref{fig:maxcusps}.  For $k < k'$ we denote by $[F_k, F_{k'}]$ the
compact submanifold of $\Mfcover$ which is the complement
of the open infinite-volume components of $\Mfcover \setminus (F_k
\cup F_{k'})$.

 \begin{figure}[bt]
\begin{center}
\begin{tikzoverlay}[scale=0.825]{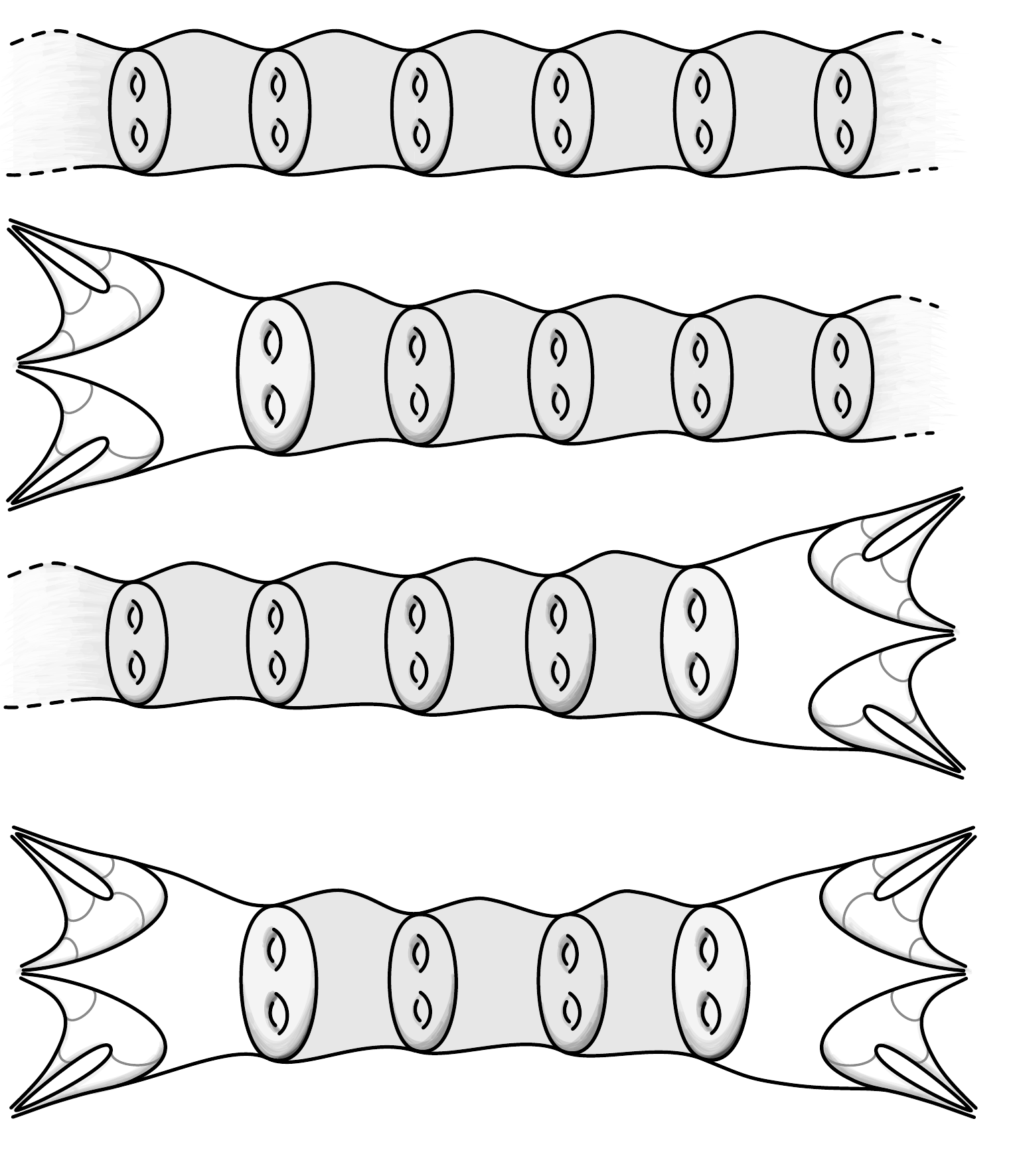}[]
    \node[left] at (0.3,105.9) {$\Mfcover$:};
    \node[left] at (4.3,85.4) {$\core(M_A)$:};
    \node[left] at (4.3,52.9) {$\core(M_C)$:};
    \node[left] at (0.0,20.1) {$Y_n:$};
    \begin{scope}[font=\small]
      \node[above=2pt] at (13.2,111.0) {$F_{-2}$};
      \node[above=2pt] at (27.0,111.0) {$F_{-1}$};
      \node[above=2pt] at (41.0,111.0) {$F_{0}$};
      \node[above=2pt] at (54.8,111.0) {$F_{1}$};
      \node[above=2pt] at (69.1,111.0) {$F_{2}$};
      \node[above=2pt] at (82.9,111.0) {\small $F_{3}$};
    \end{scope}
    \node[] at (17.5,79.7) {$A$};
    \node[above=2pt] at (26.5,86.6) {\small $F_A$};
    \node[] at (48.4,79.9) {$E_A$};
    \node[] at (48.5,53.4) {$E_C$};
    \node[above=2pt] at (68.3,60.0) {\small $F_C$};
    \node[] at (77.6,53.5) {$C$};
    \node[] at (18.0,20.0) {$A_n$};
    \node[] at (49.2,20.0) {$B_n$};
    \node[] at (79,19.6) {$C_n$};
\end{tikzoverlay}
\end{center}
\vspace{-1.3cm}
  \caption{The manifolds used in the proof of Theorem~\ref{thm:maxcusps}.}
  \label{fig:maxcusps}
\end{figure}
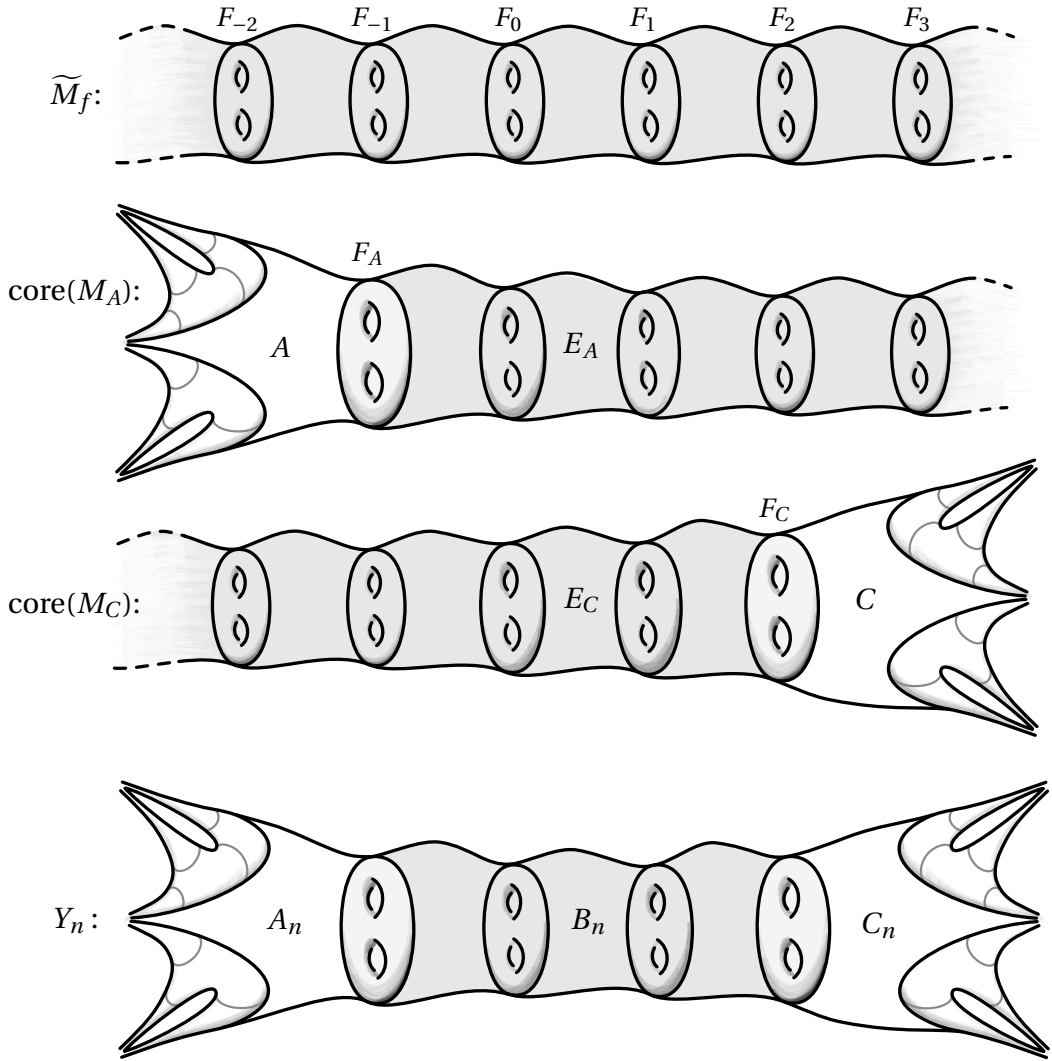  

We may consider the marking $h_k \colon S \to \Mfcover$ where $$h_k =
\alpha^k \circ h_0 \colon S \to \Mfcover.$$ Here, $h_k(S)$ is $F_k$ and as elements of $AH(S)$ we have
$$(h_k, \Mfcover) = f^k
(h_0,\Mfcover).$$

By the Bounded Geometry Theorem \cite{Minsky2001}, there is an $L$
depending only on $S$ so that for all large $n$ the manifold
$\core^0(Y_n)$ admits an $L$--bi-Lipschitz 
homeomorphism, or {\em model map},
$$\phi_n \colon [F_{-n},F_n] \to \core^0(Y_n).$$   
Since the volume of $[F_{-n},F_n]$ is $2 n \vol( M_f )$,
we have
$$\vol\left(\core^0(Y_n)\right) \to \infty  \mtext{as $n \to \infty$.}$$

The
homotopy class of $\phi_n$ is chosen so that $\phi_n \circ h_0$
corresponds to the standard marking on $Y_n$; in other words, as
elements of $AH(S)$ we have
$$(\phi_n \circ h_0, Y_n) = M\big(f^{-n}(P),f^n(P)\big).$$
For each integer $k$ with $|k| < n$, the copy of the fiber $F_k$
provides a marking for $Y_n$ via the model map $\phi_n$ by taking
$$\phi_n \compos h_k \colon S \to Y_n,$$ marked by the translate
$F_k$ in $[F_{-n},F_n]$.  Then we have
$$(\phi_n \compos h_k, Y_n) = f^k(\phi_n \compos h_0, Y_n).$$
Let
$$g_{n,k}  = \phi_n \compos h_k$$ denote this marking, and note that
$g_{n,0}$ corresponds to the standard marking of $Y_n$.

We note that for each $k$ with $|k| \le n$, the manifold $M\big(f^{-n +k}(P),
f^{n +k}(P)\big)$ is isometric to $M\big(f^{-n}(P), f^n(P)\big) = Y_n$.  In
particular, indexing the one-sided iterations by $M\big(P,f^{2n}(P)\big)$ and
$M\big(f^{-2n}(P),P\big)$ we obtain manifolds that are isometric to $Y_n$ by
the isometry $\alpha^n$ and $\alpha^{-n}$ respectively.

To prove the theorem, we start by describing $A_n$ and $C_n$.  
By
Proposition~\ref{prop:limits}, the sequences $\left\{ M\left(P, f^{2n}(P)
  \right) \right\}$ and $\left\{ M\left(f^{-2n}(P), (P) \right)
\right\}$ converge strongly to limits in $AH(S)$ with one end
asymptotically isometric to the positive end of $\Mfcover$ and the
negative end of $\Mfcover$ respectively.  The sequence
$\left\{ M\left(f^{-n}(P), f^n(P) \right) \right\}$ converges strongly
to $\Mfcover$ itself.

Let $M_A$ in $AH(S)$ be the strong limit of $M\big(P, f^{2n}(P)\big)$.
We now explain the needed decomposition of $M_A$ which is sketched in
Figure~\ref{fig:maxcusps}.  By Proposition~\ref{prop:limits} there is
an embedded surface $F_A$ in $\core(M_A)$, homotopic to the marking,
so that $F_A$ divides $\core(M_A)$ into a component $A$ with bounded
volume and an infinite-volume (neighborhood of an) end $E_A$ so that
$E_A$ is $1 + \epsilon/(2 \inj(M_f))$ bi-Lipschitz to (a neighborhood
of) the positive end of $\Mfcover$.  The finite-volume submanifold $A
\subset \core(M_A)$ has boundary
 $$ \bdry A = \bdry \core(M_A) \sqcup F_A.$$
In particular, $A$ is chosen so that we have 
\begin{equation}\label{eq:injEA}
\inj_b(M_A) > \inj(\Mfcover) - \epsilon/2 \mtext{for each $b \in
  E_A$.}
\end{equation}
We take $C$ to be the analogous subset of $M_C$, the limit
of $M\big(f^{-2n}(P), P\big)$ in $AS(S)$, cut off by a surface $F_C$;
see Figure~\ref{fig:maxcusps}.  

The intersection $A^0 = \core^0(M_A) \cap A$ being compact, the strong
convergence of $ M\left(P, f^{2n}(P) \right)$ to $M_A$ 
guarantees, for $n$ sufficiently large, smooth bi-Lipschitz embeddings 
$$\psi_{2n} \colon A^0 \to  M\left(P, f^{2n}(P) \right)$$ converging
to isometric embeddings.  We let $A_n$ be the bounded volume
submanifold of $M\left(P , f^{2n}(P)\right)$, which is isometric to
$Y_n$, cut off by the image $\psi_{2n}(F_A)$ and the convex core
boundary components corresponding to the negative end of
$M\left(P,f^{2n}(P)\right)$; compare Figure~\ref{fig:maxcusps}.
We define $C_n$ similarly and take
$$B_n = \core(Y_n) \setminus (A_n  \cup C_n).$$  

Since $\vol\big(\core(Y_n)\big)$ goes to infinity whereas 
$\vol(A_n)$ and $\vol(C_n)$ are uniformly bounded, it follows that $\vol(B_n) \to
\infty$ as $n \to \infty$, verifying the last sentence of
Theorem~\ref{thm:maxcusps}.  

We now show that for $n$ sufficiently large we have
$$\inj(B_n)> \inj(\Mfcover) -\epsilon.$$   
Assume otherwise,
and let $p_n$ be a sequence of points in $B_n$ for which
\begin{equation}
\label{bound}
\inj_{p_n}(Y_n)\le \inj(\Mfcover) -\epsilon.
\end{equation}
Then by the uniform density of the fibers $F_k$ in $[F_{-n},F_n]$ the
$L$--bi-Lipschitz model map 
$$\phi_n \colon [F_{-n}, F_n] \to \core^0(Y_n)$$
guarantees there is a sequence $\{k_n\}$ with $|k_n| < n$ so that
  $p_n$ lies at distance at most $L \cdot \diam(B_f)$ from the image 
$\phi_n(F_{k_n}) = g_{n,k_n}(S).$

The sequence $(g_{n,{k_n}},Y_n)$ 
in $AH(S)$ is represented by
remarking $Y_n$ by $f^{k_n}$.  Said differently, in $AH(S)$ we have
$$(g_{n,{k_n}},Y_n) = f^{k_n}(g_{n,0},Y_n)$$  
and $(g_{n,0},Y_n)$ represents the standard marking for which 
$$(g_{n,0},Y_n) = M\big(f^{-n}(P),f^n(P)\big).$$  

Since the basepoints $p_n$ lie at distance $L \cdot \diam(B_f)$ from
the marking surfaces $g_{n,{k_n}}(S)$, we may study the injectivity
radii at $p_n$ in terms of the limiting geometry of
$$(g_{n,{k_n}},Y_n) = 
M\big(f^{k_n -n}(P),f^{k_n +n}(P)\big).$$  
Our analysis breaks into two cases, depending on whether $n -
\abs{k_n}$ is bounded.

\emph{Case $n - |k_n|$ is unbounded.}  After passing to a subsequence
where $n - \abs{k_n} \to \infty$, Proposition~\ref{prop:limits} gives
that the sequence $(g_{n,{k_n}},Y_n)$ converges strongly to
$\Mfcover$.  As each $p_n$ lies at uniformly bounded distance of the
marking $g_{n,k_n}(S)$, there is a compact subset $K \subset \Mfcover$
and smooth embeddings $\psi_n \colon K \to Y_n$ converging to an
isometry so that $p_n \in \psi_n(K)$.

It follows that $\inj_{p_n}(Y_n) > \inj(\Mfcover) - \epsilon$ for $n$
sufficiently large contradicting assumption~(\ref{bound}).

\emph{Case $n - |k_n|$ is bounded.}  We first pass to a subsequence
where one of $n - k_n$ and $-n - k_n$ is bounded; for notational
simplicity we suppose $\abs{-n - k_n} < d$.  Then the basepoint $p_n$
lies within a uniformly bounded distance, namely $D = d \cdot L \cdot
\diam(B_f)$, of the marking surface $g_{n,-n}(S)$.  

We now employ the strong convergence of $M\big(P,f^{2n}(P)\big)$ to $M_A$. Let
$K \cong F_A \times [-1,1]$ denote a compact product neighborhood of
$F_A$ in $M_A$ containing the ball $B_{2D}(A^0)$.  By strong
convergence, we have bi-Lipschitz embeddings $\psi_n \colon K \to Y_n$
that send the neighborhood $K$ of $F_A$ to a neighborhood of the image
$\psi_n(F_A) \subset \bdry A_n$ by an orientation-preserving
diffeomorphism.  For $n$ sufficiently large, the embeddings $\psi_n$
extend to diffeomorphisms on all of $M_A$; in particular, the preimages
$\psi_n^{-1}(B_n)$ of the subsets $B_n$ lie in the positive end $E_A$
of $M_A$.

Now as each $p_n$ lies within distance $D$ of $g_{n,-n}(S)$ and the latter is
contained in $\psi_n(A^0)$, it follows that $p_n$ lies in $\psi_n(K)$ for all
large $n$. 
Our basepoints $p_n$ are in $B_n$ and hence as discussed we have
that $\psi_n^{-1}(p_n)$ lies in $E_A$.  Now by
(\ref{eq:injEA}) the injectivity radius of $E_A$ is at least
$\inj(\Mfcover) - \epsilon/2$.  Thus for large $n$ we must have
$\inj_{p_n}(Y_n) > \inj(\Mfcover) - \epsilon$ which again contradicts
assumption~(\ref{bound}).  

This shows that for sufficiently large $n$ we have $\inj_b(Y_n) >
\inj(M_f) - \epsilon$ for every $b \in B_n$, completing the proof of Theorem~\ref{thm:maxcusps}.
\end{proof}

\section{Experimental results}\label{sec:exp}

Recall that Conjecture~\ref{conj:BV} posits that for a suitable
tower $M_n$ of congruence covers of a fixed arithmetic manifold one
has 
\[
 6 \pi \cdot \frac{\log \big| H_1(M_n;\Z)_{\tors} \big|}{\vol(M_n)} \to 1.
\]
For a finite-volume hyperbolic 3-manifold (or 3-orbifold), define
\[
\TR(M) = 6 \pi \cdot \left( \frac{\log \big| H_1(M;\Z)_{\mathrm{tor}}
  \big| }{\vol(M)} \ - \ \frac{\log\big( \vol(M) \big)}{\vol(M)}\right).
\]
As the second term of $\TR(M)$ is asymptotically negligible as
$\vol(M) \to \infty$, Conjecture~\ref{conj:BV} is also equivalent to
$\TR(M_n) \to 1$.  The second term is included so that when $b_1(M) =
0$ we have that $\TR(M)$ is precisely $6 \pi \cdot \tau(M)$ by the Cheeger-M\"uller
formula (\ref{eqn:CM}).

\subsection{Twist-knot orbifolds}

First, we consider the 34 hyperbolic 3-orbifolds of Section~7 of
\cite{CalegariDunfield2006}.  These are topologically similar in that
they are all built from twist-knots, but some are arithmetic and
others are not.  As in \cite{CalegariDunfield2006}, we consider
$\Gamma_0$--type congruence covers of prime level, and 
explore what happens to $\TR(M)$ in these covers.  

Let us start with the 11 twist-knot orbifolds which are arithmetic.
Going through prime levels of norm in [500, 15{,}000] gave some 14{,}990
congruence covers of $\Gamma_0$--type, which are plotted in Figure~\ref{arithG};
as with the experiments of \cite{Sengun2011,PageBanff}, this data is
very consistent with Conjecture~\ref{conj:BV}.  Notice in
Figure~\ref{arithG} that the red dots ($b_1 > 0$) appear to be
somewhat lower (on average) than the blue dots ($b_1 = 0)$.  To
confirm this, we focus on the tail of 2{,}253 covers where $\vol(M) >
15{,}000$ and plot the distribution of $\TR$ for both types; see
Figure~\ref{arithH}.  This pattern is expected since when $b_1(M) >
0$ the analytic torsion $\tau(M)$ gets a contribution from the
regulator of $H^1(M)$; thus even if $\tau(M) \approx 1$ then $\TR(M)$
can be noticeably less than $1$.    Figure~\ref{arithB} further explores the
effect the size of $b_1$ on $\TR$.

Next, we consider the 23 twist-knot orbifolds which are
nonarithmetic.  In this case, there are some 31{,}391 congruence
covers of this type, which are plotted in Figure~\ref{nonarithG}.  Two things
are worth pointing out here.  The first is that when $b_1(M) = 0$ one
continues to have $\TR(M) \to 1$ as $\vol(M) \to \infty$, which is
strong evidence for Conjecture~\ref{conj:gentor} and also consistent
with the nonarithmetic examples of \cite{Sengun2012}.  Surprisingly,
the convergence of $\TR(M) \to 1$ appears to be
\emph{faster} than in the arithmetic case, as shown in
Figure~\ref{compH}.   The second thing is that when $b_1(M) > 0$ there
are examples where $\TR(M)$ is much less than 1 even
when the $\vol(M)$ is quite large; this suggests that
Conjecture~\ref{conj:BV} cannot be broadened to nonarithmetic
manifolds.  A more detailed look at the effect of $b_1$ on $\TR$
is given in Figure~\ref{nonarithB}.

\subsection{Covers of prime-power level}\label{sec:prime-power}

In the case of Bianchi manifolds, \Sengun\ \cite{SengunBanff}
discovered that for congruence covers of the form $\Gamma_0(\p^n)$
where $\p$ is a prime of small norm, then $\TR$ is much
smaller than in the prime-level case for covers of similar volume.
In particular, one hits a computational wall before getting convincing
evidence that $\TR \to 1$.  Here, we look at several closed arithmetic
examples which exhibit the same phenomenon; in one case, we are able
find a cover with $\TR \approx 1$ providing further evidence for
Conjecture~\ref{conj:BV}.  Part of the issue here is that these examples can have
a lot of $b_1(M)$ and hence potentially a large contribution to
$\tau(M)$ from the regulator of $H^1(M)$.

In order to tease apart the issues here, we start with some families
where $b_1(M) = 0$ for all the covers and hence $\TR(M) = 6 \pi \cdot
\tau(M)$.  Section 6.7 of \cite{CalegariDunfield2006} gives 19 closed
hyperbolic \3-manifolds (of which 3 are arithmetic) where there is
a prime $\p$ of norm 2 where the associated quaternion algebra
ramifies and moreover where $\pi_1(M)$ is 2-powerful.  Consequently,
by Theorem 6.3 of \cite{CalegariDunfield2006} the congruence covers of
level $\p^n$ all have $b_1(M) = 0$.  The data on 68 covers of these
manifolds is shown in Figure~\ref{ppower}.  The convergence of $\TR$
to $1$ seems reasonably convincing; for the 12 covers with volumes >
15{,}000, the values of $\TR$ are in $[1.000, 1.125]$.  This is still
slower than the convergence observed for covers of prime level,
especially considering that most of the manifolds here are
nonarithmetic; compare Figure~\ref{compH}.  Another arithmetic example
whose $\Gamma_0(\p^n)$--covers have $b_1 = 0$ for a prime of norm 2 is
given in Figure~\ref{cubic}; this example has the best convergence of
any tower of prime-power level that we found.  Some additional data
for other arithmetic manifolds and primes of norm 5 where again $b_1 = 0$
is given in Figure~\ref{quarticnob1}.

We turn now to five families of examples where the
$\Gamma_0(\p^n)$--covers have $b_1 > 0$ and hence the regulator term
of $\TR$ comes into play.  In each case, we start with the arithmetic
base orbifold coming from the elements of norm one in a maximal order
of a quaternion algebra $D$ over a field $K$.  The quaternion algebra
$D$ is ramified at all the real places of $K$ and at finitely many
primes of $K$ as specified in Figure~\ref{p-power-betti}.
Figure~\ref{p-power-betti} shows a marked correlation between the
amount of $b_1$ and how close $\TR$ is to $1$.   While the data is
not completely conclusive, except perhaps in the case of $M_1$, it is 
consistent with the conjecture that $\TR \to 1$.

\subsection{Computational notes}

The computations here were done with Magma \cite{Magma}.  The code for
building the covers of twist-knot orbifolds is available at
\cite{CDpaperwebsite}.  The base orbifolds for
Section~\ref{sec:prime-power} were constructed by the program
KleinianGroups \cite{KleinianGroups}.
\enlargethispage{0.5cm}
\vspace{-0.5cm}
\begin{figure}[!h]
\begin{center}
\begin{tikzoverlayabs}[width=4.0in]{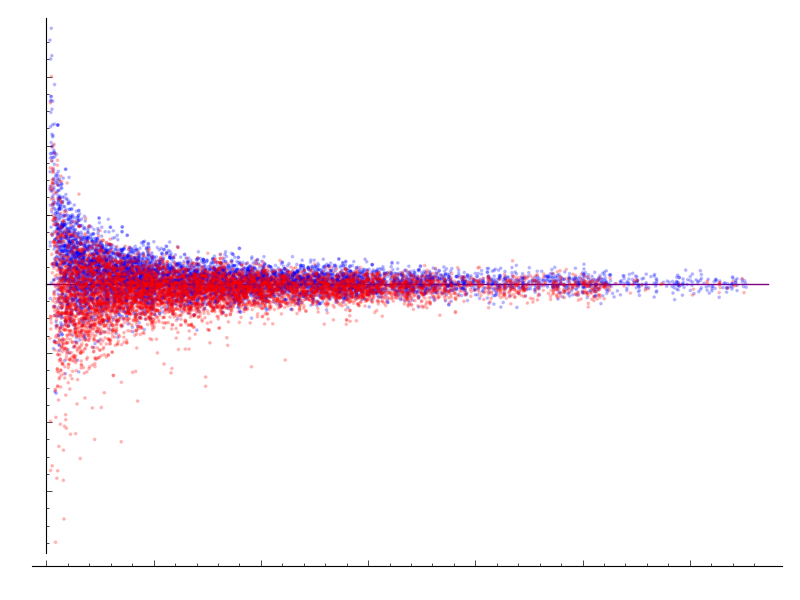}[font=\footnotesize]
\draw (0.058009281196067554, 0.046666666666666662) node[ below] {0};
\draw (0.19204596831096377, 0.046666666666666662) node[ below] {5{,}000};
\draw (0.32608265542585996, 0.046666666666666662) node[ below] {10{,}000};
\draw (0.46011934254075604, 0.046666666666666662) node[ below] {15{,}000};
\draw (0.59415602965565228, 0.046666666666666662) node[ below] {20{,}000};
\draw (0.72819271677054842, 0.046666666666666662) node[ below] {25{,}000};
\draw (0.86222940388544456, 0.046666666666666662) node[ below] {30{,}000};
\draw (0.051064836751623113, 0.18137711721535227) node[left ] {0.4};
\draw (0.051064836751623113, 0.29655019774185015) node[left ] {0.6};
\draw (0.051064836751623113, 0.41172327826834804) node[left ] {0.8};
\draw (0.051064836751623113, 0.52689635879484575) node[left ] {1};
\draw (0.051064836751623113, 0.64206943932134364) node[left ] {1.2};
\draw (0.051064836751623113, 0.75724251984784152) node[left ] {1.4};
\draw (0.051064836751623113, 0.87241560037433941) node[left ] {1.6};
\draw (0.95, 0.1) node {\small $\vol(M)$};
\draw (0.07, 0.90) node[right] {\small $\TR(M)$};
\end{tikzoverlayabs}
\end{center}
\vspace{-0.6cm}
\caption{Congruence covers of arithmetic twist-knot
  orbifolds.  The blue dots are covers where
$b_1 = 0$ and the red dots covers where $b_1 > 0$.}\label{arithG}

\vspace{-0.1cm}

\begin{center}
\begin{tikzoverlayabs}[width=4.0in]{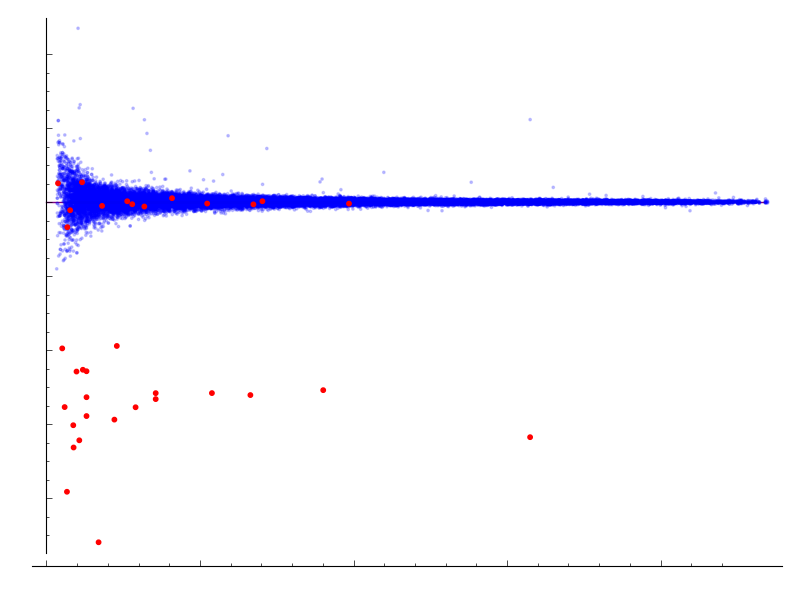}[font=\footnotesize]
  \draw (0.058009281196067554, 0.046666666666666662) node[ below] {0};
  \draw (0.24994297670598017, 0.046666666666666662) node[ below] {10{,}000};
  \draw (0.44187667221589283, 0.046666666666666662) node[ below] {20{,}000};
  \draw (0.63381036772580535, 0.046666666666666662) node[ below] {30{,}000};
  \draw (0.82574406323571814, 0.046666666666666662) node[ below] {40{,}000};
  \draw (0.051064836751623113, 0.1698634692118671) node[left ] {0.2};
  \draw (0.051064836751623113, 0.2931788960722031) node[left ] {0.4};
  \draw (0.051064836751623113, 0.41649432293253907) node[left ] {0.6};
  \draw (0.051064836751623113, 0.5398097497928751) node[left ] {0.8};
  \draw (0.051064836751623113, 0.66312517665321102) node[left ] {1};
  \draw (0.051064836751623113, 0.78644060351354717) node[left ] {1.2};
  \draw (0.051064836751623113, 0.90975603037388308) node[left ] {1.4};
  \draw (0.95, 0.1) node {\small $\vol(M)$};
  \draw (0.07, 0.90) node[right] {\small $\TR(M)$};
\end{tikzoverlayabs}
\end{center}
\vspace{-0.6cm}
\caption{Congruence covers of nonarithmetic twist-knot orbifolds; as
  before, blue dots indicate $b_1 = 0$ and red dots $b_1 > 0$.}\label{nonarithG}

\end{figure}

\begin{figure}
\begin{center}
\begin{tikzoverlayabs}[width=3in]{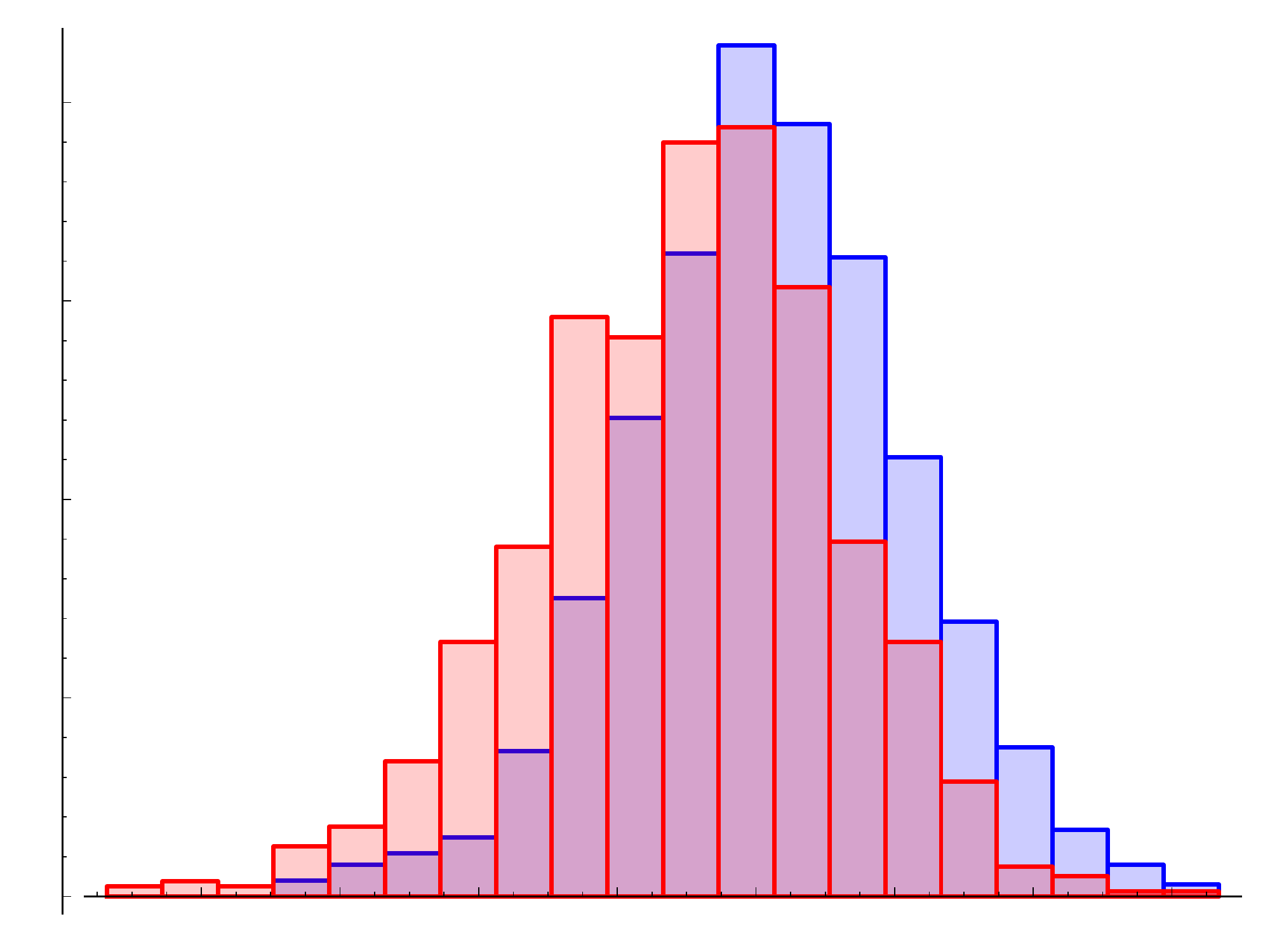}[font=\footnotesize]
  \draw (0.5, 0.0) node[below=0.05] {\small $\TR(M)$};
  \draw (0.15855822448850604, 0.049156256848564499) node[ below] {0.92};
  \draw (0.26773567095930062, 0.049156256848564499) node[ below] {0.94};
  \draw (0.37691311743009381, 0.049156256848564499) node[ below] {0.96};
  \draw (0.48609056390088767, 0.049156256848564499) node[ below] {0.98};
  \draw (0.59526801037168153, 0.049156256848564499) node[ below] {1};
  \draw (0.7044454568424755, 0.049156256848564499) node[ below] {1.02};
  \draw (0.81362290331326936, 0.049156256848564499) node[ below] {1.04};
  \draw (0.92280034978406322, 0.049156256848564499) node[ below] {1.06};
  \draw (0.042226562500000009, 0.058415516107823751) node[left ] {0};
  \draw (0.042226562500000009, 0.26691412975314577) node[left ] {5};
  \draw (0.042226562500000009, 0.47541274339846779) node[left ] {10};
  \draw (0.042226562500000009, 0.6839113570437898) node[left ] {15};
  \draw (0.042226562500000009, 0.89240997068911176) node[left ] {20};
\end{tikzoverlayabs}
\end{center}
\vspace{-0.4cm}
\caption{Histogram for $\TR(M)$ for arithmetic covers of twist-knot
  orbifolds with $\vol(M) > 15{,}000$; as before, red is $b_1 > 0$ and blue $b_1 = 0$.
}\label{arithH}

  \begin{center}
    \begin{tikzoverlayabs}[width=4.5in]{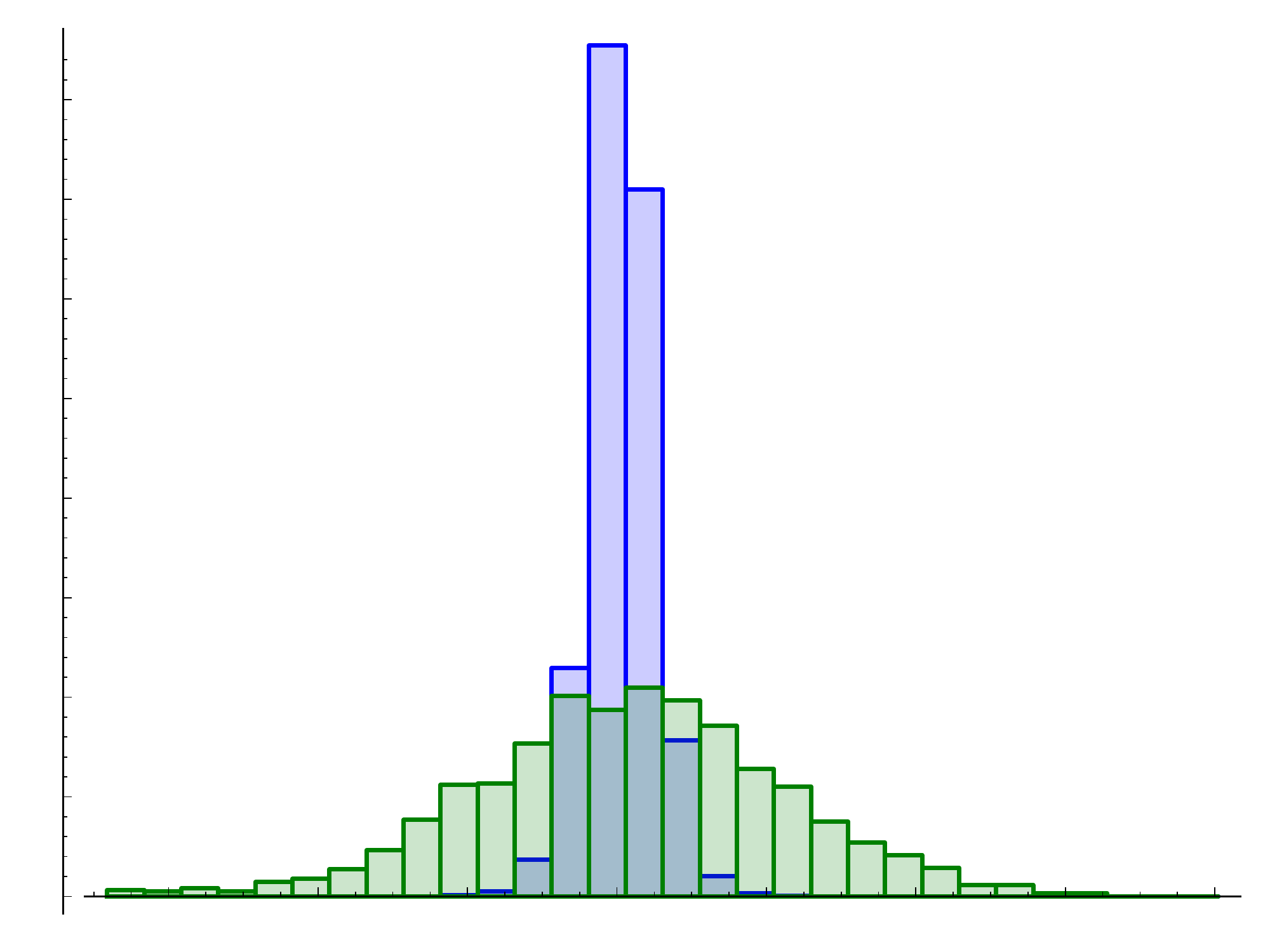}[\footnotesize]
      \draw (0.13273441558302465, 0.049156256848564499) node[ below] {0.94};
      \draw (0.25042290663300903, 0.049156256848564499) node[ below] {0.96};
      \draw (0.36811139768299345, 0.049156256848564499) node[ below] {0.98};
      \draw (0.48579988873297852, 0.049156256848564499) node[ below] {1};
      \draw (0.60348837978296366, 0.049156256848564499) node[ below] {1.02};
      \draw (0.72117687083294868, 0.049156256848564499) node[ below] {1.04};
      \draw (0.83886536188293315, 0.049156256848564499) node[ below] {1.06};
      \draw (0.95655385293291817, 0.049156256848564499) node[ below] {1.08};
      \draw (0.042617187499999987, 0.058415516107823751) node[left ] {0};
      \draw (0.042617187499999987, 0.16302416832852923) node[left ] {10};
      \draw (0.042617187499999987, 0.26763282054923465) node[left ] {20};
      \draw (0.042617187499999987, 0.37224147276994018) node[left ] {30};
      \draw (0.042617187499999987, 0.47685012499064555) node[left ] {40};
      \draw (0.042617187499999987, 0.58145877721135109) node[left ] {50};
      \draw (0.042617187499999987, 0.68606742943205645) node[left ] {60};
      \draw (0.042617187499999987, 0.79067608165276193) node[left ] {70};
      \draw (0.042617187499999987, 0.89528473387346741) node[left ] {80};
      \draw (0.5, 0.0) node[below=5pt] {\small $\TR(M)$};
    \end{tikzoverlayabs}
  \end{center}
  \vspace{-0.4cm}
  \caption{Histograms for covers where $\vol(M) > 15{,}000$.  In blue are  all the nonarithmetic covers (with two outliers removed), and in green are arithmetic covers with $b_1 = 0$.}
    \label{compH}
\end{figure}

\begin{figure}[htb]
  \begin{center}
  \begin{tikzoverlayabs}[width=4.4in]{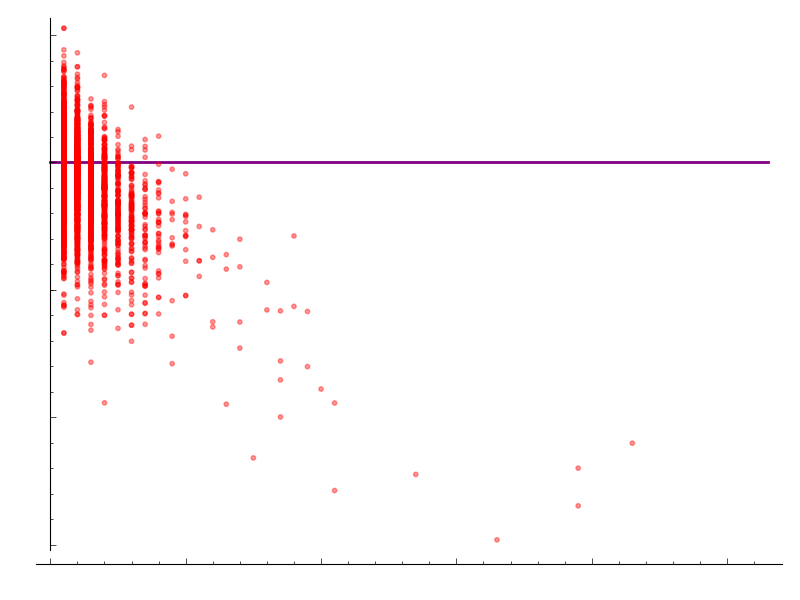}[font=\footnotesize]
    \draw (0.062941860823471404, 0.050833333333333314) node[ below] {0};
    \draw (0.23211650736001019, 0.050833333333333314) node[ below] {10};
    \draw (0.40129115389654901, 0.050833333333333314) node[ below] {20};
    \draw (0.5704658004330877, 0.050833333333333314) node[ below] {30};
    \draw (0.73964044696962661, 0.050833333333333314) node[ below] {40};
    \draw (0.90881509350616518, 0.050833333333333314) node[ below] {50};
    \draw (0.055997416379026971, 0.092380061474776912) node[left ] {0.7};
    \draw (0.055997416379026971, 0.30465331128737505) node[left ] {0.8};
    \draw (0.055997416379026971, 0.51692656109997326) node[left ] {0.9};
    \draw (0.055997416379026971, 0.72919981091257113) node[left ] {1};
    \draw (0.055997416379026971, 0.94147306072516934) node[left ] {1.1};
    \draw (0.95, 0.1) node {\small $b_1(M)$};
    \draw (0.1, 0.95) node[right] {\small $\TR(M)$};
  \end{tikzoverlayabs}
  \vspace{-0.2cm}
  \end{center}
  \caption{The relationship between $\TR(M)$ and $b_1(M)$ for covers of
   arithmetic twist-knot orbifolds where $b_1(M) > 0$.  Excludes
   covers of volume less than 5{,}000.}
 \label{arithB}

  \begin{center}
    \begin{tikzoverlayabs}[width=4.4in]{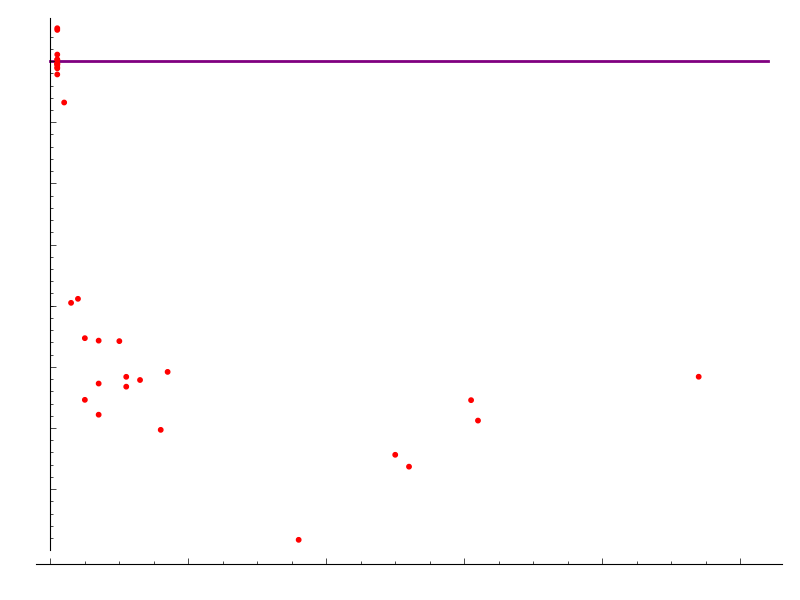}[\footnotesize]
      \draw (0.062965805384769466, 0.050833333333333314) node[ below] {0};
      \draw (0.23538929607391557, 0.050833333333333314) node[ below] {20};
      \draw (0.4078127867630616, 0.050833333333333314) node[ below] {40};
      \draw (0.58023627745220763, 0.050833333333333314) node[ below] {60};
      \draw (0.75265976814135371, 0.050833333333333314) node[ below] {80};
      \draw (0.92508325883049991, 0.050833333333333314) node[ below] {100};
      \draw (0.056021360940325032, 0.18478196729774718) node[left ] {0.3};
      \draw (0.056021360940325032, 0.28670103587709317) node[left ] {0.4};
      \draw (0.056021360940325032, 0.38862010445643919) node[left ] {0.5};
      \draw (0.056021360940325032, 0.49053917303578531) node[left ] {0.6};
      \draw (0.056021360940325032, 0.59245824161513116) node[left ] {0.7};
      \draw (0.056021360940325032, 0.69437731019447735) node[left ] {0.8};
      \draw (0.056021360940325032, 0.79629637877382342) node[left ] {0.9};
      \draw (0.056021360940325032, 0.89821544735316927) node[left ] {1};
      \draw (0.95, 0.1) node {\small $b_1(M)$};
      \draw (0.1, 0.95) node[right] {\small $\TR(M)$};
    \end{tikzoverlayabs}
    \end{center}
    \vspace{-0.2cm}
    \caption{Covers of nonarithmetic twist-knot orbifolds with $b_1 >0$. }
    \label{nonarithB}
\end{figure}

\begin{figure}
  \begin{center}
    \begin{tikzoverlayabs}[width=4.6in]{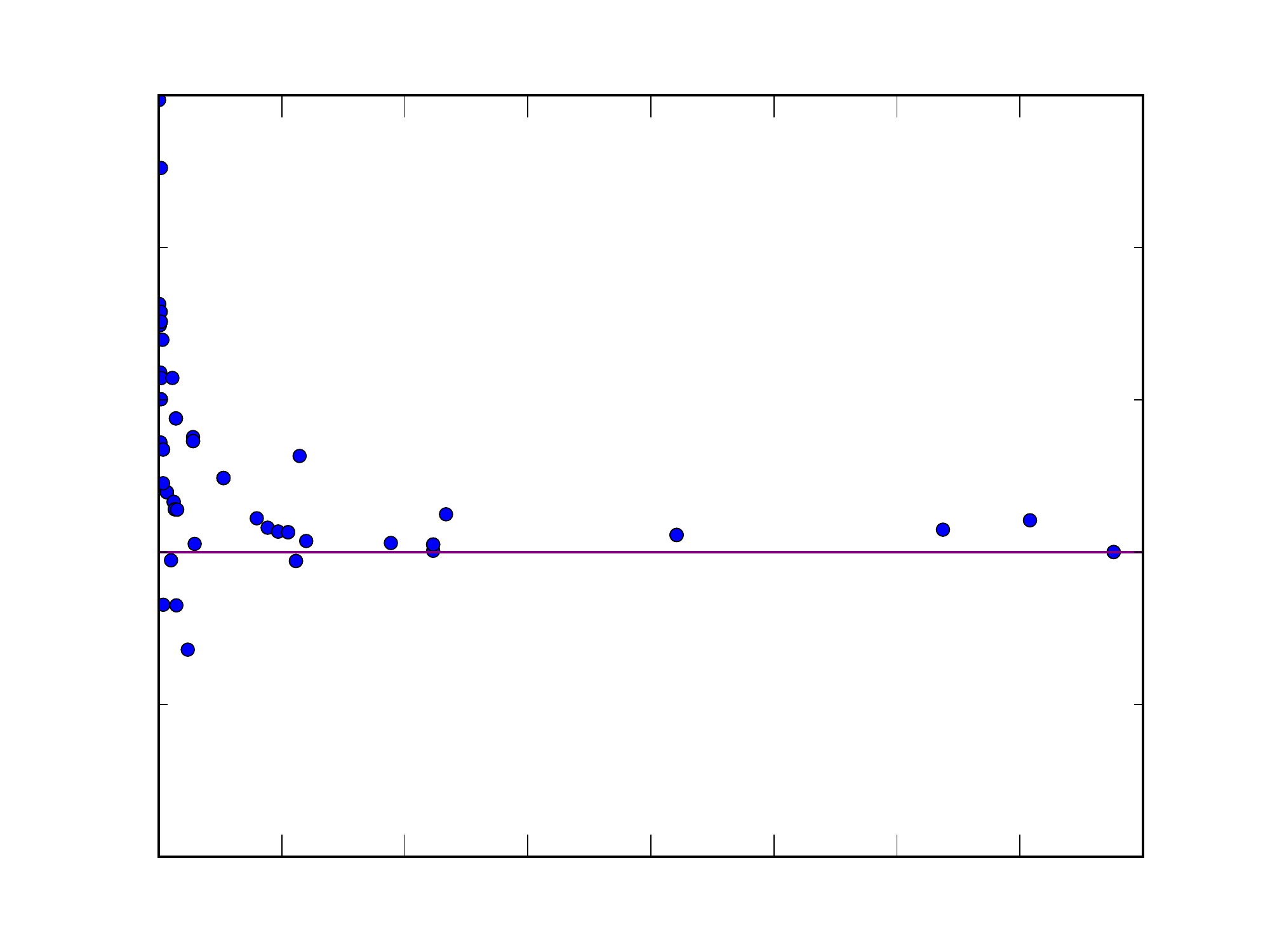}[\footnotesize]
      \draw (0.512500, 0.02) node[ below] {$\vol(M)$};
      \draw (0.125000, 0.090741) node[ below] {$0$};
      \draw (0.221875, 0.090741) node[ below] {$10{,}000$};
      \draw (0.318750, 0.090741) node[ below] {$20{,}000$};
      \draw (0.415625, 0.090741) node[ below] {$30{,}000$};
      \draw (0.512500, 0.090741) node[ below] {$40{,}000$};
      \draw (0.609375, 0.090741) node[ below] {$50{,}000$};
      \draw (0.706250, 0.090741) node[ below] {$60{,}000$};
      \draw (0.803125, 0.090741) node[ below] {$70{,}000$};
      \draw (0.900000, 0.090741) node[ below] {$80{,}000$};
      \draw (0.071094, 0.500000) node[left ] {$\TR(M)$};
      \draw (0.118056, 0.100000) node[left ] {$0.0$};
      \draw (0.118056, 0.260000) node[left ] {$0.5$};
      \draw (0.118056, 0.420000) node[left ] {$1.0$};
      \draw (0.118056, 0.580000) node[left ] {$1.5$};
      \draw (0.118056, 0.740000) node[left ] {$2.0$};
      \draw (0.118056, 0.900000) node[left ] {$2.5$};
    \end{tikzoverlayabs}
    \begin{tikzoverlayabs}[width=4.6in]{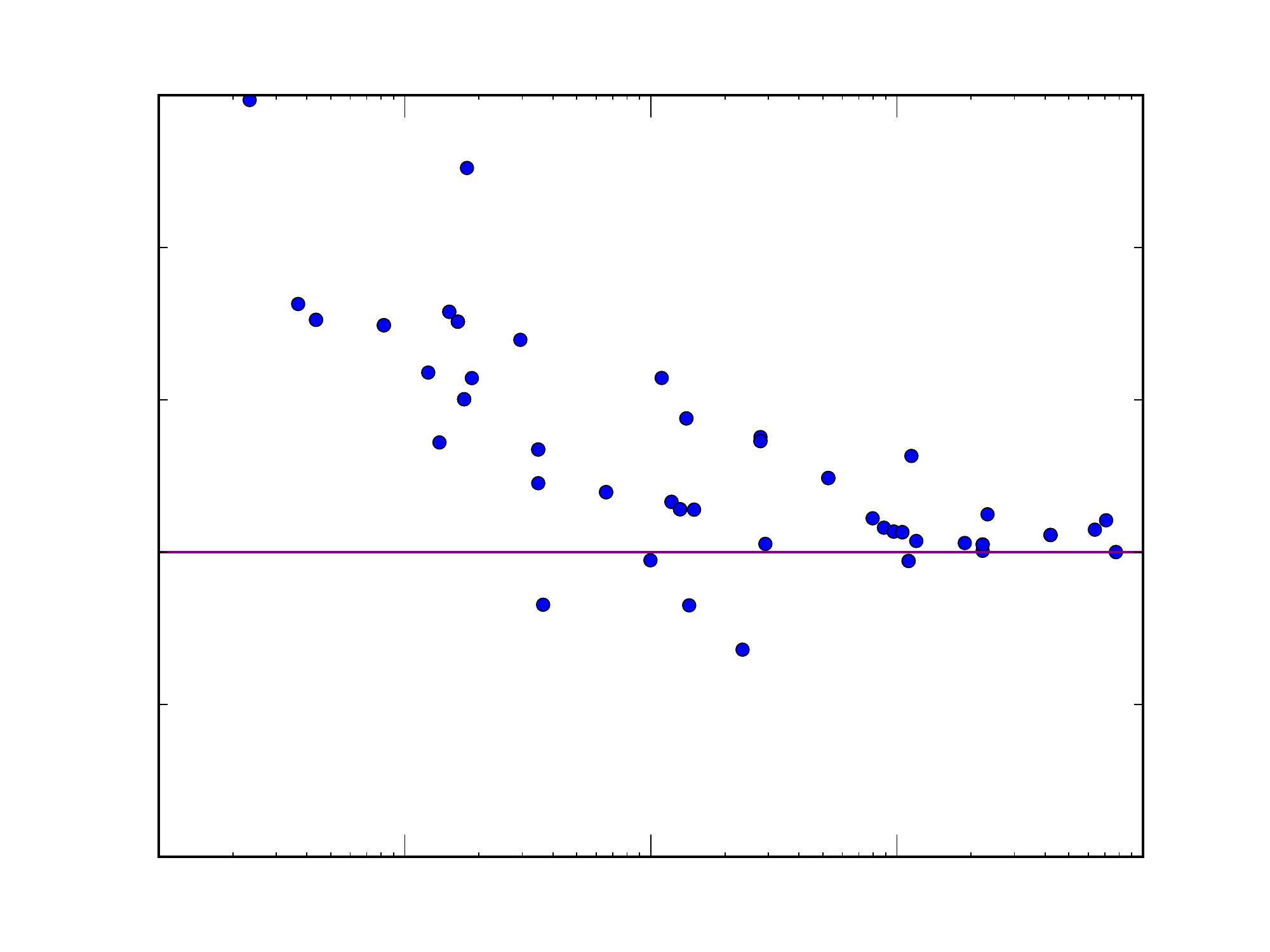}[\footnotesize]
      \draw (0.512500, 0.02) node[ below] {$\vol(M)$};
      \draw (0.125000, 0.090741) node[ below] {${10^{1}}$};
      \draw (0.318750, 0.090741) node[ below] {${10^{2}}$};
      \draw (0.512500, 0.090741) node[ below] {${10^{3}}$};
      \draw (0.706250, 0.090741) node[ below] {${10^{4}}$};
      \draw (0.900000, 0.090741) node[ below] {${10^{5}}$};
      \draw (0.071094, 0.500000) node[left ] {$\TR(M)$};
      \draw (0.118056, 0.100000) node[left ] {$0.0$};
      \draw (0.118056, 0.260000) node[left ] {$0.5$};
      \draw (0.118056, 0.420000) node[left ] {$1.0$};
      \draw (0.118056, 0.580000) node[left ] {$1.5$};
      \draw (0.118056, 0.740000) node[left ] {$2.0$};
      \draw (0.118056, 0.900000) node[left ] {$2.5$};
    \end{tikzoverlayabs}
  \end{center}

\caption{Regular congruence covers of level $\p^n$ where $N(\p) = 2$.
  The data is the same in both plots, the only difference being
  whether the volume axis has a log scale.}\label{ppower}
\end{figure}

\begin{figure}
  \begin{center}
    \begin{tikzoverlayabs}[width=4.2in]{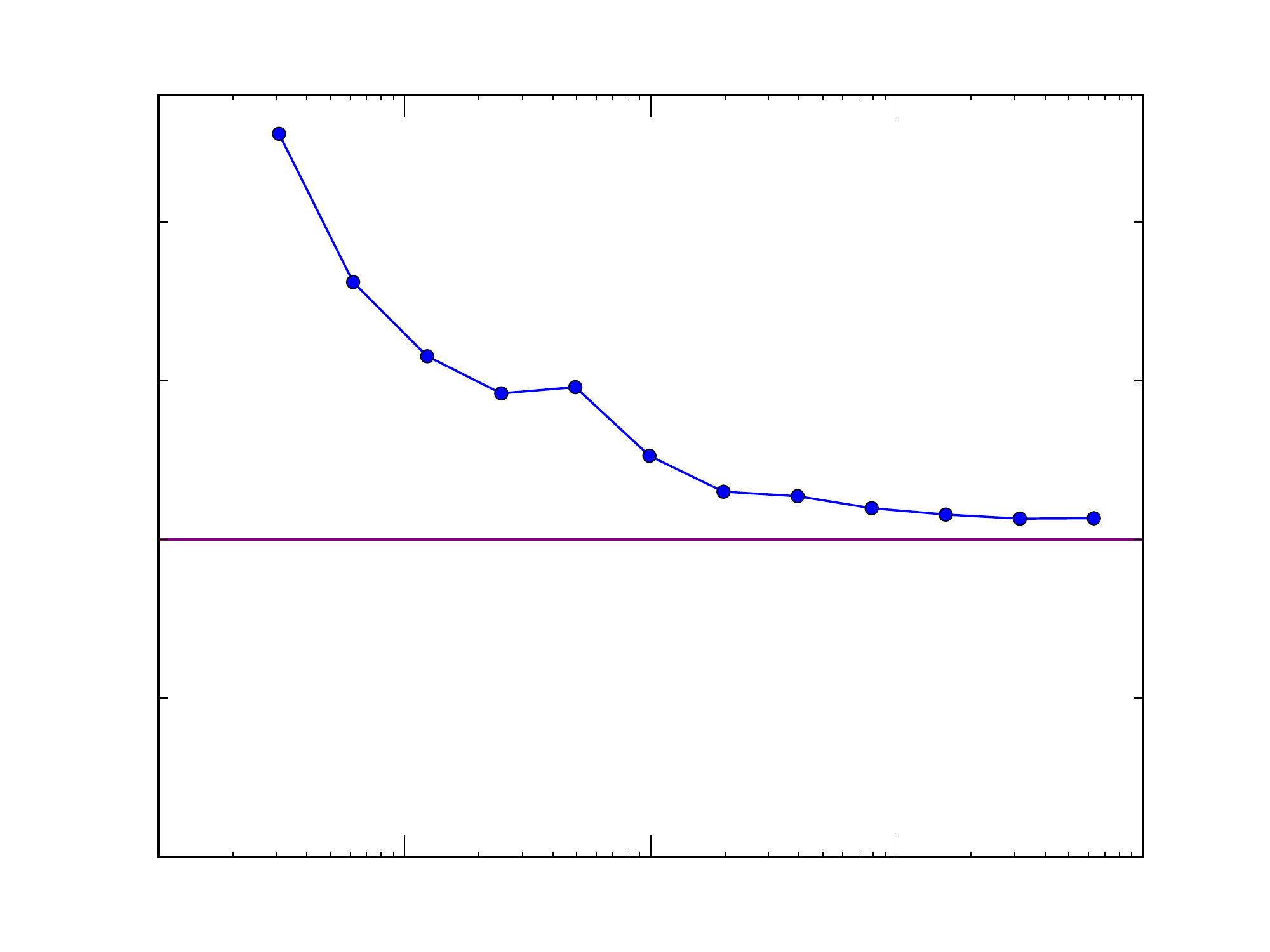}[\footnotesize]
      \draw (0.512500, 0.02) node[ below] {$\vol(M)$};
      \draw (0.125000, 0.090741) node[ below] {${10^{1}}$};
      \draw (0.318750, 0.090741) node[ below] {${10^{2}}$};
      \draw (0.512500, 0.090741) node[ below] {${10^{3}}$};
      \draw (0.706250, 0.090741) node[ below] {${10^{4}}$};
      \draw (0.900000, 0.090741) node[ below] {${10^{5}}$};
      \draw (0.098828, 0.500000) node[left ] {$\TR(M)$};
      \draw (0.118056, 0.100000) node[left ] {0.0};
      \draw (0.118056, 0.266667) node[left ] {0.5};
      \draw (0.118056, 0.433333) node[left ] {1.0};
      \draw (0.118056, 0.600000) node[left ] {1.5};
      \draw (0.118056, 0.766667) node[left ] {2.0};
  \end{tikzoverlayabs}
\end{center}


\caption{The base orbifold $M$ is arithmetic of the following form.
    The field $K$ has defining polynomial $x^3 + 2x - 1$ and the
    quaternion algebra $D$ is ramified at the real place of $K$ and
    the unique prime of norm 4. The orbifold $M$ corresponds to
    elements of norm one in a maximal order in $D$.  Congruence covers
    of are type $\Gamma_0(\p^n)$ where $\p$ is the prime of norm 2.  The
    values of $\TR$ in the tail are $< 1.07$; compare with
    Figure~\ref{compH}.}\label{cubic}
\end{figure}

\begin{figure}
  \begin{center}
    \begin{tikzoverlayabs}[width=4.4in]{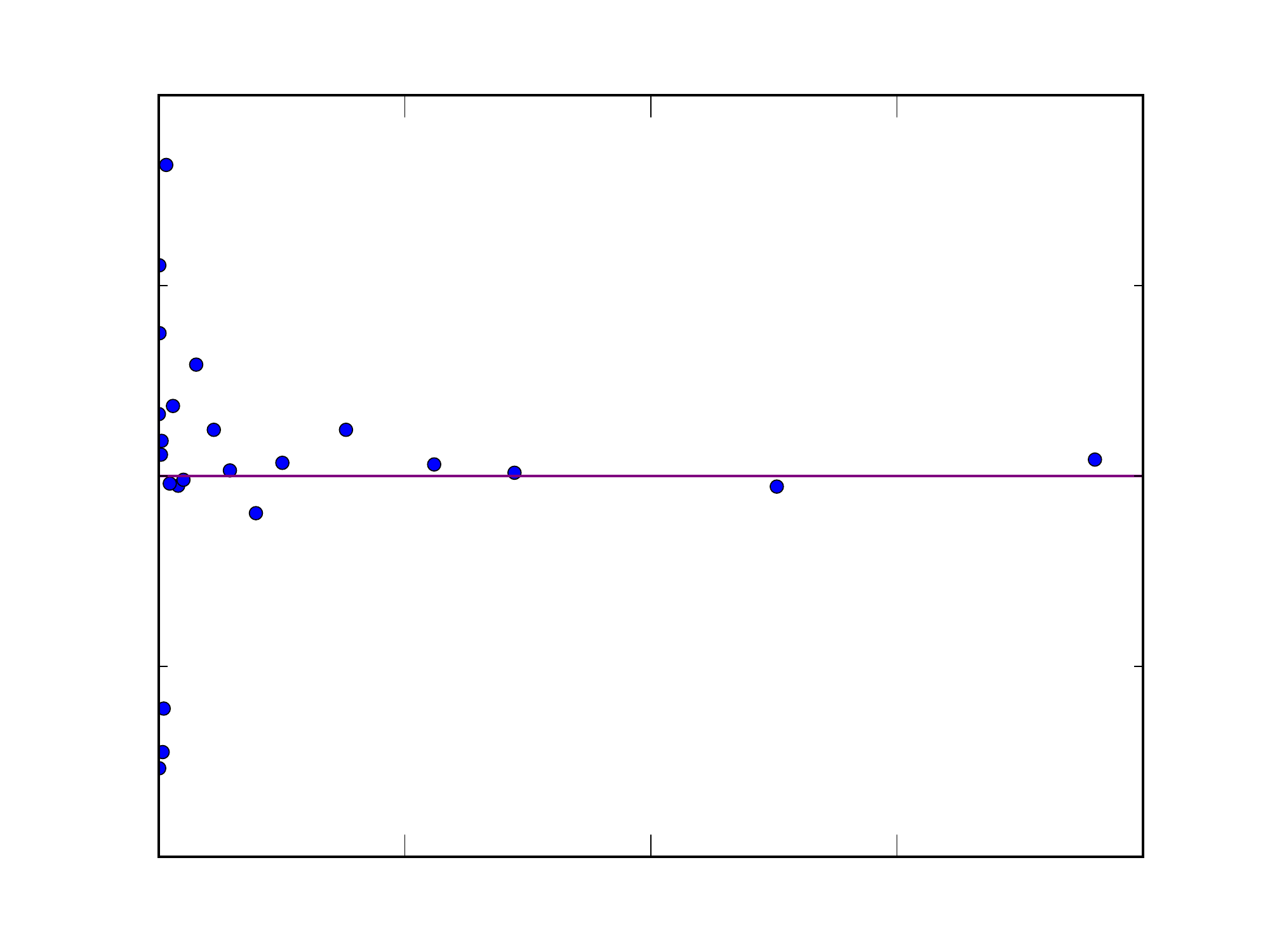}[\footnotesize]
      \draw (0.512500, 0.03) node[ below] {\small $\vol(M)$};
      \draw (0.125000, 0.090741) node[ below] {0};
      \draw (0.318750, 0.090741) node[ below] {5{,}000};
      \draw (0.512500, 0.090741) node[ below] {10{,}000};
      \draw (0.706250, 0.090741) node[ below] {15{,}000};
      \draw (0.900000, 0.090741) node[ below] {20{,}000};
      \draw (0.098828, 0.600000) node[left ] {\small $\TR(M)$};
      \draw (0.118056, 0.100000) node[left ] {0.0};
      \draw (0.118056, 0.300000) node[left ] {0.5};
      \draw (0.118056, 0.500000) node[left ] {1.0};
      \draw (0.118056, 0.700000) node[left ] {1.5};
      \draw (0.118056, 0.900000) node[left ] {2.0};
   \end{tikzoverlayabs}
   
   \begin{tikzoverlayabs}[width=4.4in]{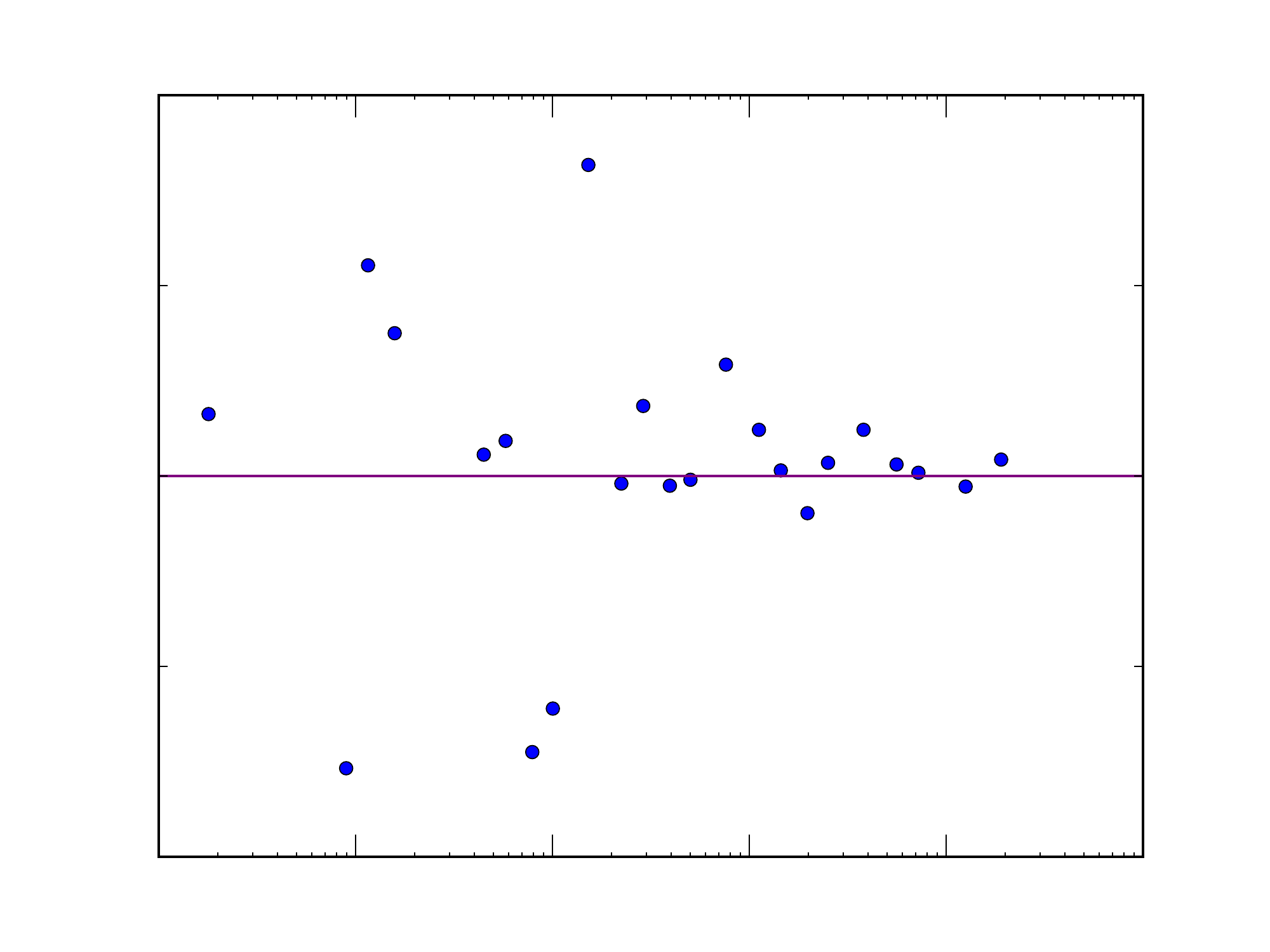}[\footnotesize]
      \draw (0.512500, 0.03) node[ below] {\small $\vol(M)$};
     \draw (0.125000, 0.090741) node[ below] {${10^{0}}$};
     \draw (0.280000, 0.090741) node[ below] {${10^{1}}$};
     \draw (0.435000, 0.090741) node[ below] {${10^{2}}$};
     \draw (0.590000, 0.090741) node[ below] {${10^{3}}$};
     \draw (0.745000, 0.090741) node[ below] {${10^{4}}$};
     \draw (0.900000, 0.090741) node[ below] {${10^{5}}$};
     \draw (0.098828, 0.600000) node[left ] {\small $\TR(M)$};
     \draw (0.118056, 0.100000) node[left ] {0.0};
     \draw (0.118056, 0.300000) node[left ] {0.5};
     \draw (0.118056, 0.500000) node[left ] {1.0};
     \draw (0.118056, 0.700000) node[left ] {1.5};
     \draw (0.118056, 0.900000) node[left ] {2.0};
   \end{tikzoverlayabs}
 \end{center}

\caption{Regular congruence covers of level $\p^n$ where $N(\p) = 5$.
  The data is the same in both plots, the only difference being
  whether the volume axis has a log scale.   The base orbifolds come
  from quaternion algebras over small quartic fields which ramify precisely
  at the two real places of the base field; all these covers have $b_1
  = 0$. 
}\label{quarticnob1}
\end{figure}

\begin{figure}
  \begin{center}
  \begin{tikzoverlayabs}[width=3.5in]{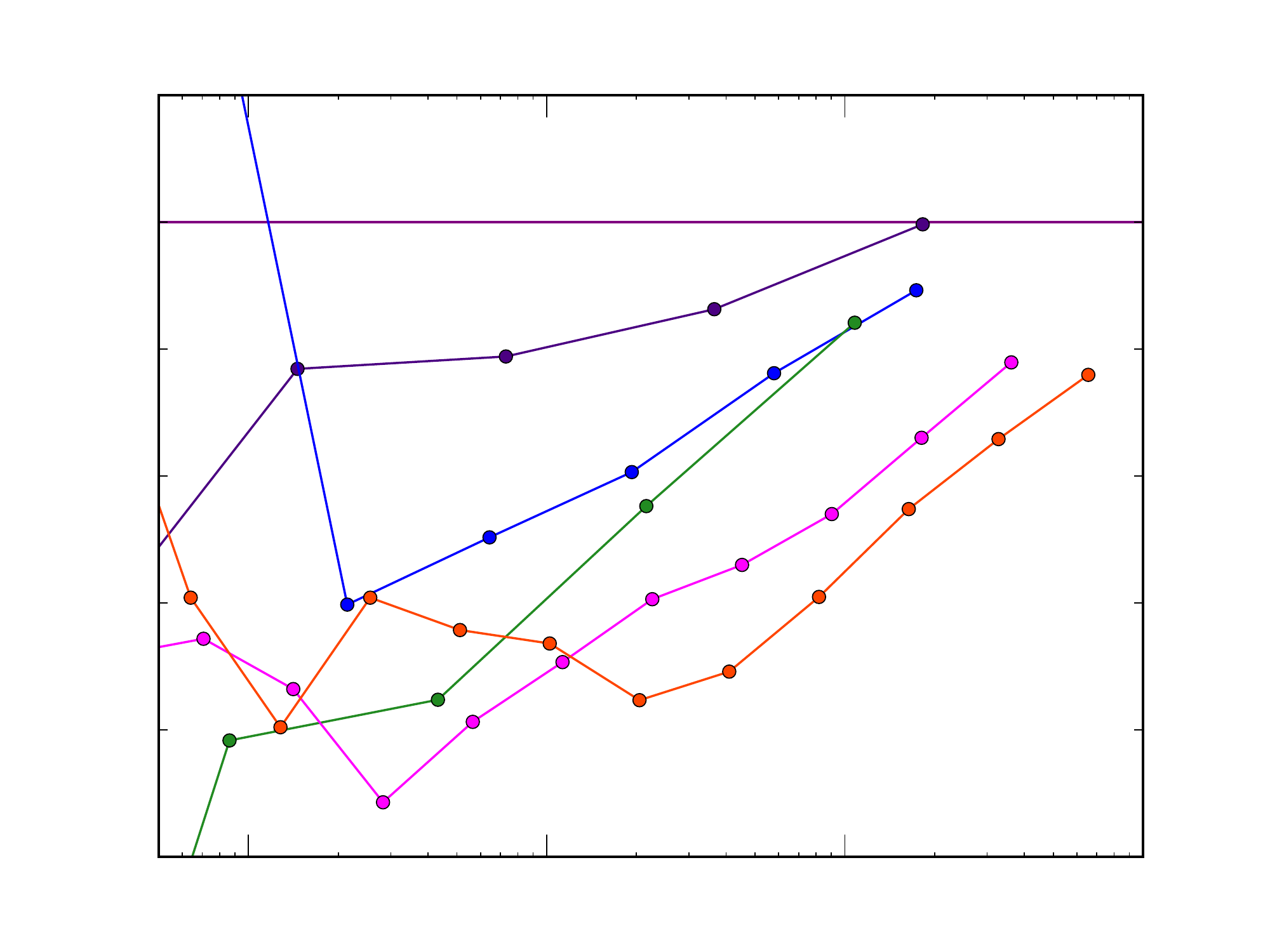}[\footnotesize]
  \draw (0.512500, 0.042083) node[below=5pt] {\small $\vol(M)$};
  \begin{scope}[xshift=2pt]
    \draw (0.195674, 0.090741) node[below] {${10^{2}}$};
    \draw (0.430450, 0.090741) node[below] {${10^{3}}$};
    \draw (0.665225, 0.090741) node[below] {${10^{4}}$};
    \draw (0.900000, 0.090741) node[below] {${10^{5}}$};
  \end{scope}
  \draw[color=white] (0.060938, 0.500000) node[left] {\small $b_1(M)$};
  \draw (0.010898, 0.600000) node[left, rotate=90] {\small $\TR(M)$};
  \begin{scope}[yshift=0.5pt]
    \draw (0.118056, 0.100000) node[left] {0.0};
    \draw (0.118056, 0.233333) node[left] {0.2};
    \draw (0.118056, 0.366667) node[left] {0.4};
    \draw (0.118056, 0.500000) node[left] {0.6};
    \draw (0.118056, 0.633333) node[left] {0.8};
    \draw (0.118056, 0.766667) node[left] {1.0};
    \draw (0.900000, 0.233333) node[right] {0.2};
    \draw (0.900000, 0.366667) node[right] {0.4};
    \draw (0.900000, 0.500000) node[right] {0.6};
    \draw (0.900000, 0.633333) node[right] {0.8};
    \draw (0.900000, 0.766667) node[right] {1.0};
  \end{scope}
  \node[above] at (0.727,0.764) {$M_1$};
  \node[right] at (0.722,0.694) {$M_2$};
  \node[below] at (0.694,0.661) {$M_3$};
  \node[above] at (0.820,0.619) {$M_4$};
  \node[below=3pt] at (0.857,0.606) {$M_5$};
\end{tikzoverlayabs}

\begin{tikzoverlayabs}[width=3.5in]{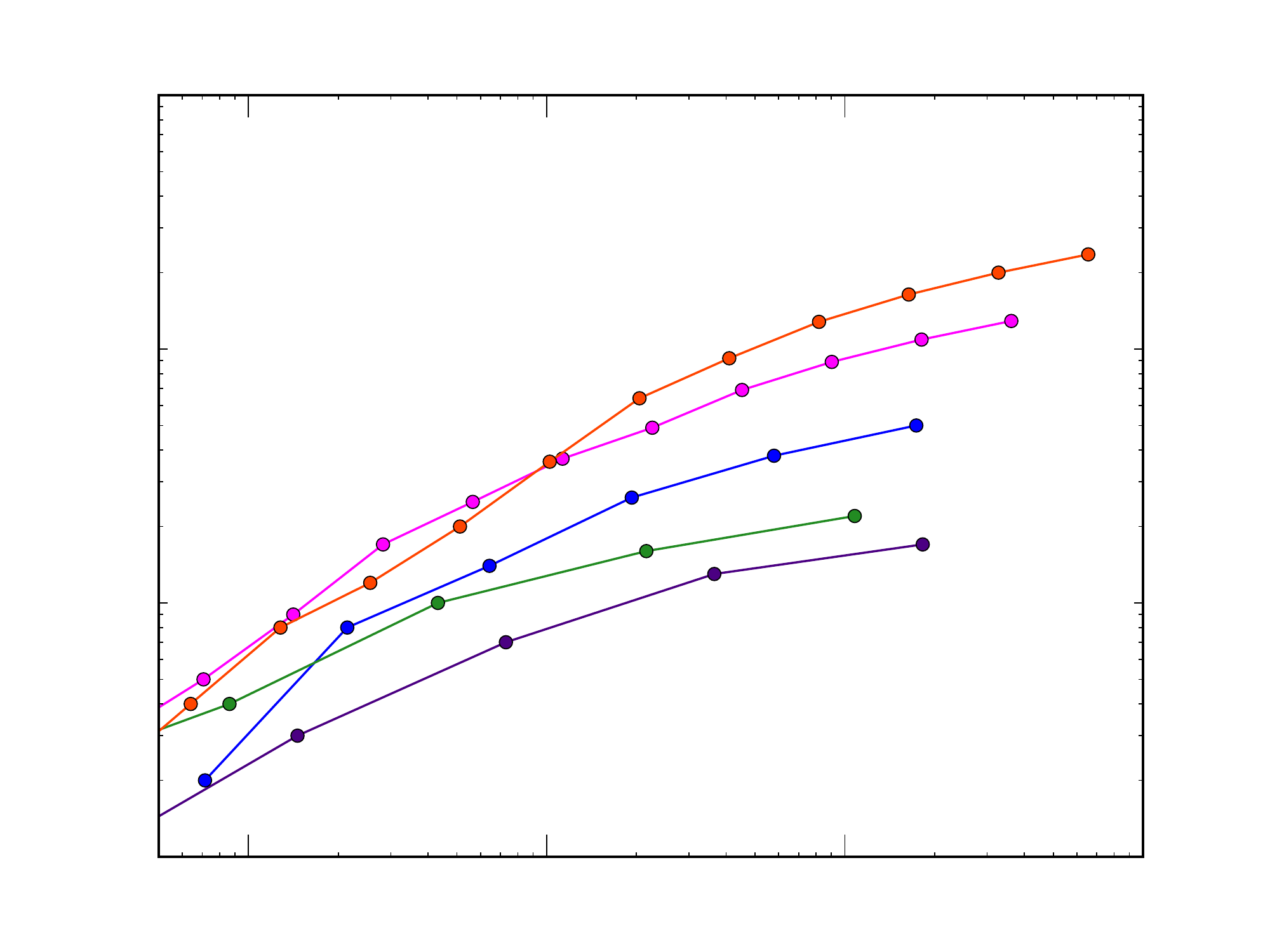}[\footnotesize]
  \draw (0.512500, 0.042083) node[below=5pt] {\small $\vol(M)$};
  \begin{scope}[xshift=2pt]
    \draw (0.195674, 0.090741) node[below] {${10^{2}}$};
    \draw (0.430450, 0.090741) node[below] {${10^{3}}$};
    \draw (0.665225, 0.090741) node[below] {${10^{4}}$};
    \draw (0.900000, 0.090741) node[below] {${10^{5}}$};
  \end{scope}
  \draw (0.060938, 0.500000) node[left] {\small $b_1(M)$};
  \begin{scope}[yshift=0.5pt]
    \draw (0.118056, 0.100000) node[left] {${10^{0}}$};
    \draw (0.118056, 0.366667) node[left] {${10^{1}}$};
    \draw (0.118056, 0.633333) node[left] {${10^{2}}$};
    \draw (0.118056, 0.900000) node[left] {${10^{3}}$};
    \draw (0.900000, 0.366667) node[right] {${10^{1}}$};
    \draw (0.900000, 0.633333) node[right] {${10^{2}}$};
    \draw (0.900000, 0.900000) node[right] {${10^{3}}$};
  \end{scope}
  \node[below] at (0.747,0.426) {$M_1$};
  \node[right] at (0.721,0.563) {$M_2$};
  \node[right] at (0.673,0.479) {$M_3$};
  \node[right] at (0.797,0.664) {$M_4$};
  \node[above] at (0.857,0.734) {$M_5$};
\end{tikzoverlayabs}

\vspace{0.4cm}
{\small
\renewcommand{\arraystretch}{1.2}
\begin{tabular}{c|l|c|c|c|c}
  & Defining poly of $K$ & $\Delta_K$ & $\mathrm{Ram}_{\mathit{finite}}(D)$ & $\p$ &
  volume \\ \hline
  $M_1$ & $x^4 - x^3 - 3 x^2 - x + 1$ & $-1323$ & $\emptyset$ & $\q_5$
  & $0.9732...$\\
  $M_2$ & $x^3 - 2x - 2$ & $-76$ & $\{\q_2\}$ & $\q_3$ & $0.6617...$ \\\hline
  $M_3$ & $x^4 - 2x^3 + 3x^2 - 1$ & $-976$ & $\emptyset$ & $\q_5$& $0.5757...$ \\ 
  $M_4$ & $x^3 - x^2 + x - 2$ &  $-83$ & $\{ \q_5 \}$ & $\q_2$ &
  $2.9435...$ \\ \hline
  $M_5$ & $x^2 - 7$ & $-7$ & $\{\q_2, \q_7\}$ & $\qbar_2$ & $5.3334...$
\end{tabular}
}
\vspace{-0.2cm}

\end{center}

\caption{ Covers of the form $\Gamma_0(\p^n)$ of the arithmetic
  orbifolds $M_n$ specified by the data in the table above,
  specifically the orbifold coming from the elements of norm one in a
  maximal order of a quaternion algebra $D$ over a field $K$.
  Here $\q_r$ denotes a prime in $\cO_K$ of norm $r$; this prime
  is unique in every case except the last example, where $\q_2$ and
  $\qbar_2$ denote the two primes in $K = \Q(\sqrt{-7})$ of norm 2.  }
\label{p-power-betti}
\end{figure}

\clearpage
{\RaggedRight 
\bibliographystyle{nmd/math} 
\bibliography{big-IHS}
}
\end{document}